\documentclass[EJP]{ejpecp} 

\usepackage{enumitem,comment}

\newtheorem{algorithm}{Algorithm}

\SHORTTITLE{Finite neutral genealogy models}

\TITLE{Takeover, fixation and identifiability in \\ finite neutral genealogy models} 

\AUTHORS{%
  Eric~Foxall\footnote{University of British Columbia, Canada.
    \EMAIL{eric.foxall@ubc.ca}}
  \and 
  Jen Labossi{\`e}re \footnote{University of British Columbia, Canada.}}

\KEYWORDS{population genetics; neutral model; lookdown representation; identifiability; size-biased ordering; coalescence; fixation; Galton-Watson tree; spinal decomposition.} 

\AMSSUBJ{60K35; 92D99} 

\SUBMITTED{February 21, 2023} 

\VOLUME{0}
\YEAR{2020}
\PAPERNUM{0}
\DOI{10.1214/YY-TN}

\ABSTRACT{
For neutral genealogy models in a finite, possibly non-constant population, there is a convenient ordered rearrangement of the particles, known as the lookdown representation \cite{donktz}, that greatly simplifies the analysis of the family trees. By introducing the dual notions of forward and backward neutrality, we give a more intuitive implementation of this rearrangement. We also show that the lookdown arranges subtrees in size-biased order of the number of their descendants, a property that is familiar in other settings \cite{bercoal} but appears not to have been previously established in this context. In addition, we use the lookdown to study three properties of finite neutral models, as a function of the sequence of unlabelled litter sizes of the model: uniqueness of the infinite path (fixation), existence of a single lineage to which almost all individuals can trace their ancestry (takeover) and whether or not we can infer the lookdown rearrangement by examining the unlabelled genealogy model (identifiability). Identifiability of the spine path in size-biased Galton-Watson trees was studied in \cite{sbsurv}, so we also discuss connections to those results, by relating the spinal decomposition to the lookdown.}

\def\lph{\alpha}

\def\bin{\mathrm{Binomial}}
\def\hyp{\mathrm{Hyp}}
\def\var{\mathrm{Var}}

\def\lf{\lfloor}
\def\rf{\rfloor}

\def\unif{\mathrm{uniform}}

\def\E{\mathcal{E}}
\def\V{\mathcal{V}}

\def\T{\mathcal{T}}
\def\ep{\epsilon}
\def\F{\mathcal{F}}

\def\nid{\noindent}

\def\tbf{\textbf}
\def\enumar{\begin{enumerate}[noitemsep,label={\arabic*.}]}
\def\enumrom{\begin{enumerate}[noitemsep,label={(\roman*)}]}
\def\enumalph{\begin{enumerate}[noitemsep,label={(\alph*)}]}
\def\enumend{\end{enumerate}}
\def\itemgo{\begin{itemize}[noitemsep]}
\def\1{\mathbf{1}}
\def\itemend{\end{itemize}}
\def\pmtx{\begin{pmatrix}}
\def\pmtrx{\end{pmatrix}}
\def\N{\mathbb{N}}
\def\R{\mathbb{R}}
\def\P{\mathcal{P}}
\def\Gma{\Gamma}
\def\gma{\gamma}

\def\eqd{\ \smash{\stackrel{\mathrm{d}}{=}}\ }
\def\Lrarrow{\Leftrightarrow}

\def\ev{\mathbb{E}}
\def\pr{\mathbb{P}}
\begin{document}

\section{Introduction}
\normalsize

This article is concerned with the structure of neutral genealogy models with finite population size. A neutral model is one in which each individual in a given generation, or at a given time, has an equal chance of reproductive success. This description applies to the two most basic population models: Galton-Watson trees, which model population growth, and the Wright-Fisher model with no selection, which models genetic inheritance. It also applies to far-reaching generalizations thereof, such as superBrownian motion and Fleming-Viot processes, provided that the rates of birth, death and replacement do not depend on the location or type of individuals. The latter restriction means that, on the other hand, something as simple as a two-type branching process generally fails to be neutral. The ``neutral'' descriptor also applies to many models of coalescence; indeed, exchangeability, which is the mathematical property that defines what it means to be neutral, is a crucial assumption in much of coalescent theory (see for example \cite{bercoal}).\\

The vast majority of existing work is focused on Markov processes with fixed transition rules. We will turn our attention to neutral models with a fixed, arbitrary initial population and sequence of (unlabelled) litter sizes (numbers of offspring of each individual), a setting that was briefly described in Section 2 of \cite{donktz}. In this setting, the multiset of litter sizes in each generation is fixed, though the assignment of litters to parents is allowed to be random. This gives a finite model which is as arbitrary as possible, subject to being neutral, and allows us to understand precisely how some basic, but key, structural properties of the model depend on its litter sizes. The first two such properties are 
\enumrom
\item \emph{takeover}: the existence of a single lineage for which the proportion of the population in generation $m$ that can trace its ancestry to any fixed individual in that lineage tends to $1$ as $m\to\infty$, and
\item \emph{fixation}: the existence of a single infinite lineage.
\enumend
The study of infinite paths is a central topic in random graphs; a few examples include the backbone in the incipient infinite cluster in two-dimensional percolation \cite{kesten}, geodesics in first-passage percolation \cite{hoffgeo}, topological ends in uniform spanning trees/forests on integer lattices \cite{pmntl} and the infinite noodle \cite{infnood}. The problem of takeover appears to have received less attention; in some sense it is a ``forward in time'' analogue of the commonly observed convergence of the set of asymptotic frequencies to a permutation of $(1,0,\dots)$ in the context of coalescent theory for exchangeable partitions of $\N$ \cite{bercoal}, and indeed it makes use of similar tools such as the coalescent time scale discussed below. In the setting of neutral genealogy models that do not go extinct in finite time, we might expect a unique infinite lineage when, in some sense, the population grows slowly enough, and similarly for takeover; this intuition is more or less correct when litter sizes are not too large.\\

We will mostly use the following three objects in our analysis:
\itemgo
\item frequency, which is the proportion of the population in generation $m$ with a given ancestor in generation $n<m$,
\item the ordered lookdown representation, which arranges the individuals in generation $n$ in size-biased order of the frequency of their descendants in generation $m>n$, and
\item the coalescent time scale, which is the time scale on which, looking backwards in time, a pair of randomly chosen individuals gains a common ancestor at rate $1$, i.e., after approximately exponentially distributed time with mean 1.
\itemend

Frequency describes the genealogical structure of the population, the lookdown organizes that structure in a nice way, and the coalescent time scale governs how quickly the size-biased ordered frequencies tend towards the takeover state $(1,0,\dots,0)$. The use of frequency in characterizing model behaviour is ubiquitous in coalescent theory \cite{bercoal}, and the coalescent time scale appears in any study of genealogy and coalescent models, since it is the natural time scale on which diffusion of frequencies occurs.\\

We will also study a third property, which is what motivated this work:
\enumrom
\item[(iii)] \emph{identifiability}: whether or not the ordering of particles given by the ordered lookdown can be inferred by examining the original (unordered) neutral model.
\enumend
The problem of identifiability has been previously studied for the spine path in size-biased Galton-Watson trees \cite{sbsurv} (see spinal decomposition below), and connects to our results as discussed below.  We are not aware of any previous work on identifiability in the context of genealogy models. To provide some motivation for identifiability, note that when there is at most one infinite lineage and the extinction times of other lineages are distinct, the lookdown can be recovered by ordering vertices according to their extinction times, as described in Theorem \ref{thm:ipasyn}. \\

\nid\tbf{Lookdown representations. }The idea of lookdown representations is to find a way of assigning dynamic (i.e., time-dependent) non-negative labels (integer- or real-valued) to the particles in a population model in such a way that the label encodes useful information about the particles' future behaviour relative to other particles. We briefly summarize our contributions to the theory; the unacquainted reader is referred to Section \ref{sec:lkdnhist} for some historical context and additional discussion.\\

In this article we use the lookdown representation that appears in \cite{donktz}, that we call the ordered lookdown (see Section \ref{sec:lkdnhist} for context and Section \ref{sec:bwdneut} for the definition). A coupling with specific conditional distributions (see Corollary \ref{cor:lkdn}) connects the lookdown to the neutral model that it represents. In \cite{donktz} this ``lookdown coupling'' is obtained by explicitly constructing a permutation-valued process that maps the lookdown to the neutral model. The construction is fairly natural once the rules are understood, though the details are a bit messy. In a bid to better understand what's at play in this correspondence, we introduce the dual notions of forward and backward neutrality (see Section \ref{sec:prelim_neut}). Forward neutrality corresponds to neutrality in the sense discussed so far, which is that assignment of litter sizes to individuals in each generation is exchangeable given the past, while backward neutrality means that assignment of sibling relationships in each generation is exchangeable given the future, a property that holds for all lookdowns discussed in this article and appears to be intrinsic to lookdowns. These notions capture the relevant features of the two types of models discussed in Sections 1-2 of \cite{donktz}: models I \& III of that reference are forward neutral, while their respective lookdown representations II \& IV are backward neutral (see ``forward/backward neutral model based on $(k_n)$'' in Sections \ref{sec:fwdneut}/\ref{sec:bwdneut} for models III \& IV). By scrambling (randomizing uniformly and independently in each generation) the vertices, we show that any forward or backward neutral model with the same litter sizes generates the same unlabelled graph (i.e., has the same distribution modulo graph isomorphism). Then, by untangling the scrambled graphs while preserving independence of certain quantities we recover the desired coupling.\\

The size-biased ordering property of the ordered lookdown is not mentioned in \cite{donktz} nor in later iterations of the lookdown \cite{ktzrdrgs}\cite{ethrdgktz}. However, this property is familiar and well-used in coalescent theory for studying exchangeable partitions of $\N$. As we show in Proposition \ref{prop:lkdnsel}, it follows easily for the ordered lookdown, using the lookdown coupling.\\

\nid\tbf{Spinal decomposition. }The spinal decomposition is a way of representing size-biased Galton-Watson (GW) trees, and other branching particle models, as the union of a distinguished path, called the spine, with other non-spine branches.  Its first use under that name appears to be in \cite{LPP} in which it's used to give concise and intuitive proofs of some results for branching processes, including the Kesten-Stigum theorem \cite{KstStg} (the so-called $L\log L$ criterion, see Section \ref{sec:sd}) that asserts that the size of a supercritical branching process rescaled by its expected value (provided the latter is finite) has a positive limit on the event of survival iff the offspring distribution $(p_k)$ has $\sum_k k\log_+(k) p_k<\infty$. Since then, the spinal decomposition has been used extensively to understand branching particle models; a few examples include generalizing the the results of \cite{LPP} to certain superprocesses \cite{LlogLsprdiff}, the law of large numbers for branching diffusions \cite{SLLNbrdiff}, the survival of branching brownian motion in a strip near criticality \cite{BBMcrit}, and identities that simplify computations, such as the ``many-to-few'' lemma \cite{harrisroberts}. For a nice introduction to the spinal decomposition in the context of branching random walks, see the St-Flour notes of Shi \cite{shisf}. In the case of Galton-Watson (GW) trees, the spinal decomposition is embedded in the ordered lookdown representation, with the spine corresponding to the base path, as described in Section \ref{sec:sd}. This connection has been mentioned before (for example Section 6.2 in \cite{donktz}, referring to the ``immortal particle'' from the paper of Evans \cite{evans}) but to our knowledge has not been rigorously proved. The proof is not difficult, and is included in Section \ref{sec:sd}. In addition, the question of identifiability of the spine was addressed in \cite{sbsurv} (Section 7.3, Question 2): its identification probability is $1$ if $\sum_k kp_k \le 1$ (subcritical and critical case) or if $\sum_k kp_k>1$ and $\sum_k k \log_+(k)p_k=\infty$; we discuss the appearance of the $L\log L$ condition in this context in Section \ref{sec:subgw}. It does not appear easy to prove this using the formulae that we obtain. On the other hand, we are able to give a meaningful interpretation of the result (see Section \ref{sec:subgw}), using the connections that we discover between identifiability, the coalescent time scale and the takeover property (see Theorem \ref{thm:tkover} and Corollary \ref{cor:baseip}).\\

\nid\tbf{Main results. }In addition to the above-mentioned contributions to the theory of lookdown representations, the main results of this article concern the aforementioned three properties and can be summarized as follows:
\enumrom
\item \textit{Takeover. }Takeover occurs almost surely if and only if infinite time passes on the coalescent time scale, see Theorem \ref{thm:tkover}. Intuitively, this is because takeover occurs when, for every $n$, the probability that a uniform random pair of particles chosen in generation $m>n$ has a common ancestor in generation $n$ tends to $1$ as $m\to\infty$. In the case of asynchronous births (at most one litter of size $>1$ in each generation) it is possible to characterize takeover in terms of the sequence of population sizes, see Theorem \ref{thm:tovar}.
\item \textit{Fixation. }Fixation is a more difficult property to characterize than takeover. We're able to show that fixation coincides with takeover when the empirical second moment of multi-child litters is bounded uniformly over generations, see Theorem \ref{thm:fix}. A full characterization of fixation remains an open question.
\item \textit{Identifiability. }Two kinds of information help to identify the vertices of the ordered lookdown: extinction time, and relative frequency, of lineages. We give formulae for the identification probability under certain conditions, and discuss cases in which that probability is equal to 1 for all vertices. In particular, we find that the base path (the unique lowest-ranked path in the lookdown) has identification probability 1 iff takeover occurs. We also conjecture a general formula for identification probability. 
\enumend

The organization of the article is discussed in Section \ref{sec:org} at the end of the introduction.

\subsection{Lookdown representations}\label{sec:lkdnhist}

In this section we give a brief overview of lookdown representations of genealogy models, beginning with a bit of history to describe their origin.\\

\nid\tbf{Population models. }Since the introduction of branching processes (BPs) and population genetics models, later work has sought to generalize them in several directions: enlarging the set of types, i.e., the set of possible values of the relevant traits (finitely many types are called multi-type BPs), allowing arbitrary lifespan and offspring distribution as a function of an individual's age (CMJ BPs), randomizing the environment, adding spatial structure, and taking large-population limits. With respect to the last two points, a good deal of effort has centered on so-called superprocesses, or measure-valued processes, in which the measure represents the population density of individuals as a function of type and/or physical space (see for example the St-Flour notes of Dawson \cite{dawsonstflr} or Etheridge \cite{ethsf} for an introduction). These are generally understood (and derived) as rescaled large-population limits of models in which the time scale of mutation, selection and population growth and decay is slow enough relative to the time scale of replacement, either direct replacement or indirect replacement via birth and death events that are balanced on average, so that the structure of the latter, in the limit, is not destroyed. In particular, such models must be approximately neutral, at the time scale of replacement, in order for the measure-valued limit to exist. Measure-valued generalizations of branching processes are known variously as superBrownian motion or Dawson-Watanabe superprocesses, and of the Wright-Fisher or Moran models are known as Fleming-Viot processes.\\

\nid\tbf{Particle representations and the lookdown. }Early developments in the theory of measure-valued population models centered on the use of generators, and characterized the relevant processes as solutions of the corresponding martingale problems. Eventually, other approaches emerged that made it possible to incorporate the genealogy directly in the model and thereby establish certain properties with greater ease than was possible with older methods. One such approach is the infinite particle representation of the Fleming-Viot process that appears in \cite{donktz0}. It is based upon an ordered rearrangement of the Moran model, which is a simple (continuous-time) neutral model with fixed population size $N$ in which, for each ordered pair of particles, at rate $1/2$ the first particle is removed and the second splits into two particles, keeping the population size constant. In the ordered arrangement, each particle is given a rank, and for each unordered pair, at rate 1 the particle with the higher rank is always removed and that rank is assigned to the offspring of the particle with the lower rank. Since, in the ordered arrangement, the higher rank always ``looks down'' to the lower rank to determine its parent, it is known as the lookdown representation. Both models are shown in Figure \ref{fig:mrngrph} (the lookdown ). 

\begin{figure}
\begin{center}
\includegraphics[width=5in]{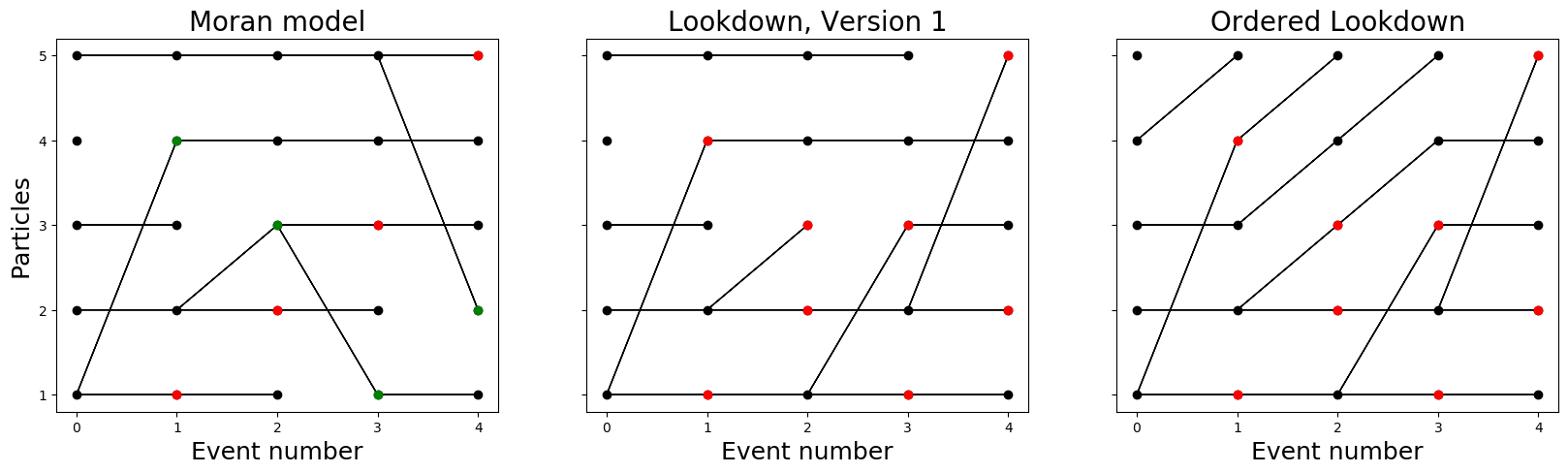}
\end{center}
\caption{Moran model and both versions of the lookdown model, with 5 particles and 4 events corresponding to the same particle pairs in each case. In the Moran model, ordered pairs are red and green, with red replacing green. In the lookdown, both are red, with the replacement rule as specified in the discussion.}
\label{fig:mrngrph}
\end{figure}

The lookdown has three very nice properties:
\enumar
\item It is consistent, i.e., for any $N$ the restriction of the $N$-particle lookdown representation to the lowest $M<N$ ranks is a copy of the $M$-particle lookdown representation,
\item It generates the same lineages as the original model, i.e., for any $N$ the $N$-particle Moran model can be obtained from the lookdown representation by a specific randomization of the particle ranks, and
\item The ranking is conditionally uniform given the past, i.e., with respect to the coupling given by the randomization mentioned above, conditioned on the Moran model on $[0,t]$, the particle ranks are uniformly distributed.
\enumend
Property 1 is not hard to deduce from the definition, and it implies the existence of an infinite-particle lookdown representation. As discussed in the introduction, properties 2-3, the lookdown coupling, take some effort to prove. Properties 2-3 should seem a bit remarkable in light of the fact that low-ranked particles in the lookdown have priority for reproduction and thus have larger than average clades (defining a clade to consist of a particle and its descendants). On the other hand, it's reasonable to expect some clades to be larger than others, simply by random chance, so perhaps the lookdown can be interpreted as a ranking of particles by the future success of their clade; this relates to the question of identifiability discussed below.\\

\nid\tbf{The ordered lookdown. }In another article that appears a few years later \cite{donktz}, the authors present a modified lookdown representation that extends more easily to neutral models with variable population size, and allows them to obtain an infinite particle representation of superBrownian motion. In the case of the Moran model it works as follows: instead of removing one of the two selected particles, ranks are reassigned while preserving their original order, and the particle with the topmost rank is removed. This also satisfies properties 1-3, and its generalization, which can be applied to arbitrary neutral models, follows the same approach: given a sequence of litter sizes, assign litters uniformly at random, order them by their lowest-ranked particle, then assign parents in that order. This ``ordered lookdown'' representation is the one that we study in this article. It has some additional properties that make it convenient to use:
\enumrom
\item[4.] The rule for assigning parents to litters extends to later generations: if $\Xi_{n,m}$ partitions generation $m$ by ancestors in generation $n<m$ and $B_1,\dots,B_{\ell}$ orders $\Xi_{n,m}$ by least element then the parent of each $B_i$ has rank $i$ (see Lemma \ref{lem:ordle}).
\item[5.] Frequencies follow a size-biased ordering: if $X_{n,m}(i)$ denotes the number of descendants, in generation $m$, of the vertex with rank $i$ in generation $n$, then $(X_{n,m}(i))_{i=1}^{X_n}$ is in size-biased order (see Proposition \ref{prop:lkdnsel}).
\enumend

Property 5 should be familiar to anyone who has studied exchangeable partitions of $\N$, see for example the first chapter of \cite{bercoal}, but appears not to have been previously established for finite genealogy models. It is not hard to prove -- it essentially follows by combining properties 3 and 4 -- but it's quite useful. Property 4 implies the extinction time of a clade is non-increasing with rank, since if the particle with rank $i$ has a descendant in generation $m$ then so does the particle with rank $i-1$. In particular, if fixation holds, and if at most one particle is removed per generation, then particle ranks are determined by the extinction time of their clades, and in particular, in the coupling from property 2 it's possible to identify the ranks of particles by looking far enough into the future. This begs the following more general question: under what conditions can we reliably identify the ranks of particles by examining their descendants? This is the question of identifiability. It is closely related to the question of takeover, and can be mostly answered with the help of properties 4 and 5.\\

\nid\tbf{Continuum lookdowns. }Since \cite{donktz}, there have been further iterations on the lookdown representation. In particular, the replacement of integer-valued ranks with real-valued levels has allowed for some interesting generalizations, including to some non-neutral models in which birth and death rates depend on a particle's type and can vary over time. A first attempt in this direction appears in \cite{ktzspatvar}; as noted by its author in a later co-written article \cite{ktzrdrgs} on the same topic, the former is limited in scope and does not provide much useful insight (which is good to know, since we found the exposition in \cite{ktzspatvar} very difficult to parse). In \cite{ktzrdrgs}, branching processes, and more general branching Markov processes, are equipped with level-based representations, in which levels take values in a fixed interval $[0,r]$, and at each time $t$, are conditionally independent uniform on $[0,r]$, given the underlying process on the time interval $[0,t]$, mirroring property 3 above. The level of each particle evolves over time according to differential equations whose terms are related to the birth and death rates of the model. Their framework also allows them to handle more complex processes such as branching processes in random environments. They are able to obtain diffusion and measure-valued limits of these various processes by taking $r\to\infty$ and by taking initial levels to be a Poisson process on $\R_+$, and the analogue of property 3 for those limits is that the levels are conditionally Poisson, rather than conditionally uniform, with intensities given by the underlying particle measure. The analogues of properties 1-3 are established respectively as follows:
\enumar
\item Consistency for different $r$ follows again from the construction of the level-based processes.
\item To show the underlying process is obtained when levels are ignored, the generator of the joint process is integrated against the independent uniform (Poisson for $r=\infty$) distribution.
\item The conditionally uniform (Poisson) property follows from the above generator calculation together with an abstract ``Markov mapping'' theorem that guarantees existence of a process with the desired property. The latter is related to filtering, see for example \cite{ktzflt} \cite{ktznpp}. 
\enumend

Further work on these continuum lookdowns appears in \cite{ethrdgktz}, in which recipes are provided for building continuum lookdown representations of models with any combination of several mechanisms of birth, death, replacement, immigration and movement. They also work several examples. As they note in their introduction, the main novelty of that article is the incorporation of so-called event-based (as opposed to individual-based) mechanisms, in which multiple individuals are simultaneously affected and multiple levels need to be adjusted in a way that respects the aforementioned conditionally uniform/conditionally Poisson structure. \\

The transition rules for levels in continuum lookdowns are reminiscent of the (integer-ranked) ordered lookdown; in particular, levels of offspring at birth are above the parent's level, particles are removed once they go above level $r$, and even the birth rate and offspring level as a function of parent level, and the movement rate given by the differential equations, seem to reflect the way that ranks evolve, on average, in the analogous ordered lookdown, if ranks are mapped onto $[0,r]$ by $i\mapsto ri/X_n$ where $i$ is rank and $X_n$ is population size in generation $n$. It would be interesting to know whether these continuum lookdowns satisfy property 5, i.e., give a size-biased or similar ordering of particles by the number of descendants they have at later times.



\subsection{Organization of the article}\label{sec:org}

The rest of the article is organized as follows. In Section \ref{sec:prelim} we introduce the main definitions and concepts. Genealogy models are introduced in Section \ref{sec:gengrph} and some examples given in Section \ref{sec:basex}. Forward and backward neutral models are defined in Sections \ref{sec:fwdneut} and \ref{sec:bwdneut}, respectively, and a general construction given for a forward or backward neutral model based on a fixed sequence of litter sizes $(k_n)$, see Definition \ref{def:kn} for the meaning. The martingale property of frequencies in forward neutral models is discussed in Lemma \ref{lem:xmart}, and the remainder of Section \ref{sec:prelim_neut} is devoted to coupling results, including Corollary \ref{cor:lkdn}, the lookdown coupling. In Section \ref{sec:sbool} we give background on size-biasing and state the size-biasing property of the ordered lookdown in Proposition \ref{prop:lkdnsel}. In Section \ref{sec:ident_def} we define identification probability, in Section \ref{sec:sd} we introduce the spinal decomposition of GW trees and in Section \ref{sec:coalts} we define the coalescent time scale for genealogy models. Section \ref{sec:tkover} contains the main result on takeover, Theorem \ref{thm:tkover}, with Section \ref{sec:asyn} specializing that result to asynchronous models. Section \ref{sec:subfix} discusses fixation. Section \ref{sec:ident} concerns identifiability. Section \ref{sec:subgw} discusses connections with the spinal decomposition. The remaining sections contain proofs of the more lengthy results and are organized according to their titles.

\section{Preliminaries}\label{sec:prelim}

\subsection{Graphs and models}\label{sec:gengrph}

\nid\tbf{Generational graphs. }A directed graph $G=(V,E)$ is a generational graph if there is a partition $(V_n)_{n\ge 0}$ of $V$ ($n$ is the generation number) such that
\enumrom
\item $E\subset \bigcup_{n\ge 0}V_n\times V_{n+1}$ and 
\item every vertex in $\bigcup_{n\ge 1}V_n$ has in-degree $1$.
\enumend
This holds iff $G$ is a directed forest, i.e., when there are no directed cycles and every vertex has in-degree at most $1$, in which case $V_0$ are the vertices with in-degree $0$ and for each $n>0$, $V_n$ are the vertices at graph distance $n$ from $V_0$.\\

Generational graphs encode lineages of individual-based population models. In discrete-time models, each time step can be viewed as one generation, while in continuous-time models, each event that changes the population can be counted as one generation, with individuals unchanged by that event advanced to the next generation by a single directed edge.\\

\nid\tbf{Visualization. }$X_n:=|V_n|$ is the population size in generation $n$ and $\tau:=\inf\{n\colon X_n=0\} \in \N\cup \{\infty\}$ is the extinction time. To picture a generational graph it's convenient to represent $V_n$ by the set $\{(n,i)\colon i \in \{1,\dots,X_n\}\}$, so that the generation $n$ increases horizontally, and each $V_n$ is a vertical stack of points with height $X_n$ (see Figure \ref{fig:gnlgrph} for an example). In particular, $V$ is determined by the sequence $(X_n)$. For convenience, $V_n$ will often be identified with the set $\{1,\dots,X_n\}$.\\

\begin{figure}
\begin{center}
\includegraphics[width=3in]{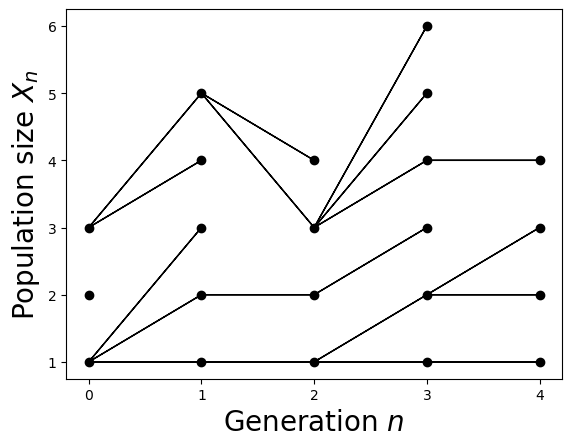}
\end{center}
\caption{Example of genealogy graph with $(X_0,\dots,X_4)=(3,5,4,6,4)$, and litter sizes $K_0=(3,0,2)$, $K_1=(1,1,0,0,2)$, $K_2=(2,1,3,0)$ and $K_3=(1,2,0,1,0,0)$.}
\label{fig:gnlgrph}
\end{figure}

\nid\tbf{Ancestors, descendants. }In a generational graph, for $n<m$ each vertex $v\in V_m$ has a unique ancestor $a_n(v)\in V_n$ (called the parent, and written as $a(v)$, if $m=n+1$) obtained by tracing backward along directed edges; for $v\in V_n$ the set
\[D_m(v):=\{w\in V_m\colon a_n(w)=v\}\]
of vertices in $V_m$ with ancestor $v$ are the \emph{descendants} of $v$ in generation $m$, and the \emph{children} if $m=n+1$. Following \cite{bercoal}, for $n<m$, define the \emph{ancestral partition} $\Xi_{n,m}$ to be the partition of $V_m$ induced by the equivalence relation $\sim_n$ defined by $v \sim_n w$ iff $a_n(v)=a_n(w)$ i.e., if $v,w \in V_m$ have the same ancestor in $V_n$.\\

\nid\tbf{Litter size, frequency. }For any $n$, $v \in V_n$ and $m\ge n$ let $X_m(v)=|D_m(v)|$ be the number of descendants of $v$ in generation $m$, called the \emph{litter size} if $m=n+1$ and written as $K_n(i)=X_{n+1}((n,i))$ for $(n,i)\in V_n$. For $m$ such that $X_m>0$, define the \emph{frequency} $x_m(v):=X_m(v)/X_m$ of $v$ in generation $m$.\\

\nid\tbf{Genealogy model. }For the purpose of this article, a genealogy model is a probability measure on the set of generational graphs.\\

\nid\tbf{Isomorphism, unlabelled graph. }We'll use the usual notion of isomorphism for directed graphs: a pair of generational graphs $(V,E)$ and $(W,F)$ are \emph{isomorphic} if there is a bijective function $f:V\to W$ such that $F=\{(f(v),f(w))\colon (v,w)\in E\}$. If $(V,E)$ and $(W,F)$ are isomorphic and $V,W$ are represented as in the visualization above then $W=V$. Given a generational graph $(V,E)$, which we think of as being labelled by $V_n=\{(n,i)\colon i\in \{1,\dots,X_n\}\}$, the corresponding unlabelled graph is the isomorphism class of $(V,E)$, or equivalently, the set of graphs $\{(V,\sigma(E))\colon \sigma \in \P(V)\}$ where $\P(V)$ is the set of permutations of $V$ that map each $V_n$ to itself and $\sigma(E)=\{(\sigma(v),\sigma(w))\colon (v,w) \in E\}$.\\

\nid\tbf{Scrambling. }Say that a random variable $\sigma = (\sigma_n)_{n<\tau} \in \P(V)$ is uniform if $(\sigma_n)$ are independent and each is a uniform random permutation of $V_n$. For a generational graph $(V,E)$ define the \emph{scrambled} graph by $(V,\sigma(E))$ where $\sigma \in \P(V)$ is uniform. The scrambled and unlabelled graphs are equivalent in the following sense: if $\sigma\in \P(V)$ is uniform, then $(V,E)$ and $(V,F)$ are isomorphic iff $\sigma(E)$ and $\sigma(F)$ have the same distribution. Scrambling is a convenient way to verify that two genealogy models have the same (in distribution) unlabelled graphs.

\subsection{Basic examples}\label{sec:basex}
Recall that $X_n$ is the number of vertices in generation $n$, and use $\E_n$ to denote the random edge set between vertices in generations $n$ and $n+1$. To set the context we begin by constructing the three models from the Introduction: Galton-Watson tree, Wright-Fisher model and Moran model.

\enumrom
\item \emph{Galton-Watson tree.} Fix a distribution on $\N$, i.e., a sequence $(p_k)_{k \in \N}$ with $p_k\ge 0$ and $\sum_k p_k=1$. Let $(K_{n,i})_{n,i\in\N}$ be a doubly-indexed i.i.d.~sequence with $P(K_{n,i}=k)=p_k$, let $X_0=1$ and recursively for $n>0$,
\[X_{n+1}=\sum_{i=1}^{X_n}K_{n,i}\]
then let $c_{n,0}=0$, $c_{n,i}=\sum_{j=1}^i K_{n,j}$ and let
\begin{align}\label{eq:Endet}
\E_n=\{((n,i),(n+1,j))\colon 1\le i\le X_n, \ c_{n,i-1}< j \le c_{n,i}\}.
\end{align}

Note that vertex $(n,i)$ has $K_{n,i}$ children. Following \cite{LPP}, children of vertices in $V_n$ are arranged in contiguous blocks in $V_{n+1}$ in the same order as their parents, forming a planar graph.\\

\item \emph{Wright-Fisher model. }Fix $N$, let $X_n=N$ for all $n$ and with $(a_{n,i})_{n,i \in \N}$ i.i.d.~uniform on $\{1,\dots,N\}$, $n\ge 1$, let
\[\E_n=\{((n,a_{n,j}),(n+1,j))\colon 1 \le j \le N\}.\]

\item \emph{Moran model. }Fix $N$, let $X_n=N$ for all $n$ and with $((a_n,c_n)\colon n\in \N)$ i.i.d.~uniform pairs on $\{1,\dots,N\}$ sampled without replacement, let
\begin{align}\label{eq:EnMrn}
\E_n=\{((n,i),(n+1,i))\}_{i=1}^N\cup\{((n,a_n),(n+1,c_n))\}\setminus \{((n,c_n),(n+1,c_n))\}.
\end{align}
\enumend

\subsection{Neutral models}\label{sec:prelim_neut}

In this section we present a novel approach to the lookdown coupling (properties 2 and 3 in Section \ref{sec:lkdnhist} and Corollary \ref{cor:lkdn} below) that uses two concepts that, to our knowledge, are new: forward and backward neutral, that were discussed in the introduction and are given in Definitions \ref{def:fwd} and \ref{def:back}. We begin with a precise definition of the unlabelled litter sizes.

\begin{definition}[unlabelled litter sizes, model based on $(k_n)$]\label{def:kn}
For $n<\tau$, let $k_n$ be equal to the litter sizes $K_n$ arranged in decreasing order. Then $(k_n)$ is the sequence of \emph{unlabelled litter sizes}.\\

If $\tau\in \N\cup\{\infty\}$ and $(k_n)_{n<\tau}$ is a deterministic sequence for which the number $X_n$ of entries in $k_n$ satisfies $X_{n+1}=\sum_{i=1}^{X_n}k_n(i)$ for $0<n<\tau$, where $X_\tau:=0$, then $(\V,\E)$ is a genealogy model \emph{based on} $(k_n)$ if it has unlabelled litter sizes $(k_n)$.
\end{definition}

One could think of each $k_n$ as a multiset, or as a list modulo permutation of the entries. Unless otherwise stated, for most results we'll assume we're working with a model based on some fixed $(k_n)$. This is a convenient assumption when permuting genealogy graphs, essentially because (i) $(k_n)$ determines the sequence $(X_n)$ of population sizes, and thus the vertex set $V$, and (ii) $(k_n)$ is invariant under permutation. Moreover, for what we do it loses no generality, since in a given model we can condition on the values of $(k_n)$.\\

The importance of $(k_n)$ in the context of neutral models is that it completely determines the distribution of the unlabelled graph; this generalizes property 2 of Section \ref{sec:lkdnhist} and is articulated in Theorem \ref{thm:neut}. The conditionally uniform property, which is property 3 in Section \ref{sec:lkdnhist}, is described in Theorem \ref{thm:fwdback}.

\subsubsection{Forward neutral}\label{sec:fwdneut}

As is nicely expressed on page 4 of \cite{ethsf}, a model is neutral if ``everyone has an equal chance of reproductive success''. Mathematically, this should be true if the list of litter sizes $(K_n(i))_{i=1}^{X_n}$ in each generation $n$, conditioned on the past, i.e., on $(X_m)_{m\le n}$ and $(\E_m)_{m<n}$, is exchangeable, i.e., has the same distribution as $(K_n(\sigma(i)))_{i=1}^{X_n}$ for every permutation $\sigma$ of the set $\{1,\dots,X_n\}$. For the formal definition we condition on the sequence of unlabelled litter sizes. Let $k_n$ be the vector $K_n$ arranged in decreasing order.

\begin{definition}\label{def:fwd}
A genealogy model is \emph{forward neutral} if, conditioned on the sequence of unlabelled litter sizes $(k_n)_{n\in \N}$, for each $n$, $K_n$ is exchangeable and independent of $(\E_m)_{m<n}$.
\end{definition}
This definition captures neutrality in the sense described above; the term forward neutral is used to distinguish it from a related notion, backward neutral, that holds for the lookdown.\\

We take a moment to show the basic examples are forward neutral.

\enumrom
\item \emph{Galton-Watson tree.} Given $n$, condition on $(K_{m,i}\colon m\ne n, \ i\in \N)$ and on $k_n$, which determines $(X_m)_{m\in \N}$, $(k_m)_{m\in \N}$ and $(\E_m)_{m\ne n}$. Since $(K_{n,i})_{i=1}^{X_n}$ are i.i.d., the probability that $K_{n,i}=k_n(\sigma(i))$ for all $i\in \{1,\dots,X_n\}$ does not depend on the choice of permutation $\sigma$.\\

\item \emph{Wright-Fisher model. }In this model $(X_m)$ is deterministic and $(\E_m)$ is an i.i.d.~sequence. Since permuting the entries in $K_n$ does not change $k_n$, it suffices to show that for each $n$, $K_n$ is exchangeable. But
\[K_n(i)= \#\{j\in\{1,\dots,N\}\colon a_{n,j}=i\}\]
and $(a_{n,j})_{i=1}^N$ are i.i.d.~uniform on $\{1,\dots,N\}$.\\

\item \emph{Moran model. }The same argument implies that it suffices to show that for each $n$, $K_n$ is exchangeable. This follows from the fact that $(a_n,c_n)$ is i.i.d.~uniform and
\[K_n(i)=\begin{cases} 2 & \text{if} \ i=a_n, \\
0 & \text{if} \ i=c_n, \\
1 & \text{otherwise}.\end{cases}\]

\enumend

A forward neutral model can be constructed by randomizing the parents and assigning children deterministically. The following is model III in \cite{donktz}.\\

\nid\tbf{Forward neutral model based on $(k_n)$. }
Let $\tau,(k_n),(X_n)$ satisfy the conditions in Definition \ref{def:kn} and let $(\sigma_n)$ be an independent sequence, with each $\sigma_n$ uniformly distributed over the set of permutations of $\{1,\dots,X_n\}$. Let $V_n=\{(n,i)\colon i=1,\dots,X_n\}$, let $K_{n,i} = k_n(\sigma_n(i))$ and define $\E_n$ from $(K_{n,i})$ as in \eqref{eq:Endet}. Then $(V,\E)$ is a forward neutral model based on $(k_n)$.\\

For example, if $k_0=(3,2,0)$, $k_1=(2,1,1,0,0)$, $k_2=(3,2,1,0)$ and $k_3=(2,1,1,0,0,0)$ and $\sigma_n^{-1}(j)$ is coloured red, green, blue for $j=0,1,2,$ respectively then we obtain the graph in Figure \ref{fig:fwdgrph}, which is just Figure \ref{fig:gnlgrph} but with colours added.\\

\begin{figure}
\begin{center}
\includegraphics[width=3in]{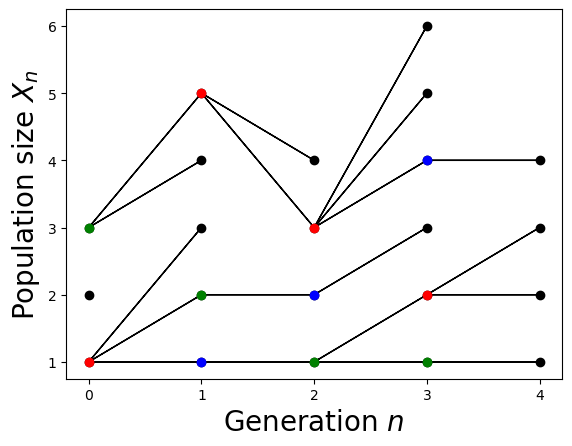}
\end{center}
\caption{Example of forward neutral model based on $k_0=(3,2,0)$, $k_1=(2,1,1,0,0)$, $k_2=(3,2,1,0)$ and $k_3=(2,1,1,0,0,0)$.}
\label{fig:fwdgrph}
\end{figure}

\nid\tbf{Rearrangement. }There is a fairly strong equivalence between forward neutral models with the same unlabelled litter sizes: any one can be rearranged into any other in an adapted way, with respect to the filtration determined by edges in past generations and some auxiliary coin flips. To discuss this property we need to talk about coupling.\\

\nid\textit{Coupling. }Following the usual definition in probability theory, a \emph{coupling} of two genealogy models is a single probability space $(\Omega,\F,P)$ and random variables $(\V,\E)$ and $(\V',\E')$ whose marginal distributions are the respective models. 

\begin{definition}[Permutation coupling]\label{def:coup}
A coupling of genealogy models $(\V,\E)$ and $(\V',\E')$ is a \emph{permutation coupling} if $\V=\V'$ almost surely and if there is a random variable $\sigma \in \P(\V)$ such that $\E'=\sigma(\E)$.
\end{definition}
Abusing terminology somewhat, in such examples we'll refer to $\sigma$ as the permutation coupling.

\begin{lemma}[Rearrangement]\label{lem:fwdcoup}
Let $(V,\E)$ and $(V,\E')$ be forward neutral models based on the same $(k_n)$ and let $U=(U_n)$ be an i.i.d.~$\unif[0,1]$ sequence independent of $\E$. There is a permutation coupling $\E'=\alpha(\E)$ such that for each $n$, $(\alpha_m)_{m\le n}$ is determined by $(\E_m)_{m<n}$ and $(U_m)_{m \le n}$.
\end{lemma}

Lemma \ref{lem:fwdcoup} is important in establishing the forward-backward and lookdown couplings below, respectively Theorem \ref{thm:fwdback} and Corollary \ref{cor:lkdn}.\\

\nid\tbf{Frequency. }Recall $X_m(v),x_m(v)$ are the number, respectively frequency of descendants of $v$ in generation $m$. If each individual in a forward neutral model has an equal chance of success, the frequencies should be martingales, which they are.

\begin{lemma}[Frequency]\label{lem:xmart}
If $(V,\E)$ is forward neutral, then for each $n$ and $v\in V_n$, $(x_m(v))_{m=n}^{\tau-1}$ is a martingale with respect to $(\F_{m-1})_{m=n}^{\tau-1}$, where $(\F_n)_{n\ge 0}$ is the natural filtration of $(\E_n)_{n\ge 0}$. If $\tau=\infty$ then $x_m(v)$ converges a.s. and in $L^1$~as $m\to\infty$.
\end{lemma}
\begin{proof}
Recall that $D(v)$ are the descendants of $v$, and let $D_m(v)=D(v)\cap V_m$ be the descendants of $v$ in generation $m$. Since $K_m$ is exchangeable and independent of $(\E_j)_{j< m}$, for each $i\in\{1,\dots,X_m\}$, $\ev[K_m(i) \mid (\E_j)_{j<m}] = (1/X_m)\sum_i k_m(i) = X_{m+1}/X_m$. Then,
$$\ev[X_{m+1}(v) \mid (\E_j)_{j<m}] = \sum_{i \in D_m(v)} \ev[K_m(i) \mid (\E_j)_{j<m}] = X_m(v)\frac{X_{m+1}}{X_m}.$$
It follows after dividing by $X_{m+1}$ that $\ev[x_{m+1}(v) \mid (\E_j)_{j<m}] = x_m(v)$, as desired. Almost sure and $L^1$ convergence follows from the martingale convergence theorem and the fact that $x_m(v)\in [0,1]$ for all $m$.
\end{proof}

\nid\tbf{Limiting frequency. }Let $x(v)$ denote $\lim_{m\to\infty} x_m(v)$, given by Lemma \ref{lem:xmart} when $\tau=\infty$. We call $x(v)$ the \emph{limiting frequency} of $v$; since $\sum_{v \in V_n}x_m(v)=1$ for $m\ge n$, $L^1$ convergence implies that $\sum_{v\in V_n} x(v)=1$.

\subsubsection{Backward neutral}\label{sec:bwdneut}

In a forward neutral model, litter sizes are exchangeable given the past. A natural ``dual'' notion is to make litters, i.e., sibling relationships, exchangeable given the future. For each $n$ let $\Xi_n$ be the partition of $\{1,\dots,X_{n+1}\}$ induced by the relation $i \sim j \Lrarrow a((n+1,i))=a((n+1,j))$.

\begin{definition}\label{def:back}
A genealogy model is \emph{backward neutral} if, conditioned on the sequence of unlabelled litter sizes $(k_n)_{n\in \N}$, for each $n$, $\Xi_n$ is exchangeable and independent of $(\E_m)_{m>n}$.
\end{definition}

\enumrom
\item \emph{Moran lookdown, version 1. }With $N$, $(X_n)$ as in the Moran model and with $((a_n,c_n)\colon n\in \N)$ i.i.d.~uniform \emph{ordered} pairs on $\{1,\dots,N\}$ with $a_n<c_n$, define $\E_n$ as in \eqref{eq:EnMrn}.\\

\item \emph{Moran lookdown, version 2. }With $N$, $(X_n)$ and $((a_n,c_n)\colon n\in \N)$ as above, let
\[\E_n=\{((n,i),(n+1,i))\}_{i=1}^{c_n-1}\cup \{((n,a_n),(n+1,c_n))\} \cup\{(n,i),(n+1,i+1))\}_{i=c_n}^{N-1} .\]
\enumend

In both versions, $(\E_n)$ is an i.i.d.~sequence and for each $n$,
\[\Xi_n=\{\{i\}\}_{i\notin \{a_n,c_n\}} \cup \{\{a_n,c_n\}\}\]
is exchangeable, which implies backwards neutrality. \\

A backward neutral model can be constructed by randomizing the siblings and assigning parents deterministically, as follows.\\

\nid\tbf{Backward neutral model based on $(k_n)$. }Let $\tau,(k_n),(X_n)$ satisfy the conditions in Definition \ref{def:kn}. A backward neutral model with unlabelled liter sizes $(k_n)$ can be constructed by randomizing sibling relationships and assigning parents deterministically. Let $\xi_n$ be a partition of $\{1,\dots,X_{n+1}\}$ with block sizes given by the non-zero entries of $k_n$; for example, if $k_n(i)>0$ for all $i$ then take
\begin{align}\label{eq:xin}
\xi_n = \{ \{1,\dots,k_n(1)\}, \{k_n(1)+1,\dots,k_n(2)\}, \dots , \{k_n(X_n-1)+1,\dots,k_n(X_n)\}\}.
\end{align}
Let $(\sigma_n)$ be independent, with each $\sigma_n$ a uniform random permutation of $\{1,\dots,X_n\}$ and let $\Xi_n=\{\sigma_n(A)\colon A \in \xi_n\}$. Then $(\Xi_n)$ are i.i.d.~and each $\Xi_n$ is exchangeable, so any choice of $\E_n$ that is a deterministic function of $\Xi_n$ and satisfies
\[\{D_{n+1}((n,i))\}_{i=1}^{X_n}=\Xi_n\]
gives a backward neutral model with unlabelled litter sizes $(k_n)$.\\

\nid\tbf{Ordered lookdown. }The specific choice of deterministic function that generalizes Moran lookdown version 2, and gives the ordered lookdown, is Model IV in \cite{donktz}. It is obtained by first arranging $\Xi_n$ in increasing order of least element, denoting this $(B_1,\dots,B_{\ell_n})$, where $\ell_n = \#\{i\colon k_n(i) \ne 0\}$, then assigning parents in that order, i.e., taking
$$\E_n =\bigcup_{i=1}^{\ell_n} \{((n,i),(n+1,j))\colon j \in B_i\}.$$

For example, with $k_0,\dots,k_3$ as in Figure \ref{fig:fwdgrph} and $B_1,B_2,B_3$ in each generation coloured red,green,blue respectively, we obtain the graph in Figure \ref{fig:bwdgrph}.\\
 
\begin{figure}
\begin{center}
\includegraphics[width=3in]{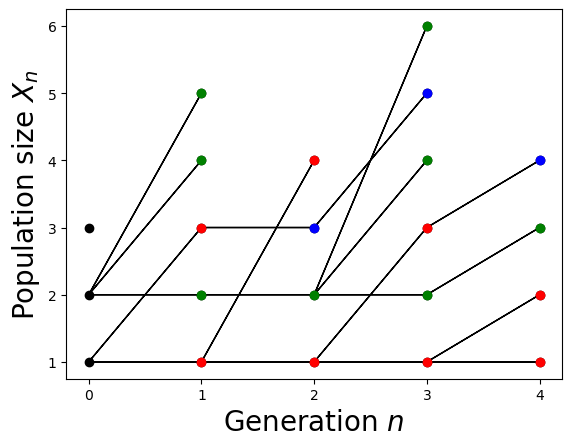}
\end{center}
\caption{Example of the ordered lookdown based on $k_0=(3,2,0)$, $k_1=(2,1,1,0,0)$, $k_2=(3,2,1,0)$ and $k_3=(2,1,1,0,0,0)$.}
\label{fig:bwdgrph}
\end{figure}

Note that vertices $(n,i)$, $i\in\{\ell_n+1,\dots,X_n\}$ have no children, and that a child cannot have a lower rank than its parent, since if $j\in B_i$ then by definition $\min B_1<\dots <\min B_i \le j$ implying $j\ge i$. In particular, $(n+1,1)$ has parent $(n,1)$ for all $n<\tau-1$, creating a path of edges called the \emph{base path}. As explained by Algorithm \ref{alg:sbo2} in Section \ref{sec:sbool}, the process of randomizing sibling relationships, then sorting by least element, leads to a size-biased ordering of litter sizes, meaning that low-ranked vertices tend to have larger litters.

\subsubsection{Forward/backward coupling, the lookdown coupling}

Forward and backward neutral models with the same unlabelled litter sizes share a connection analogous to properties 2-3 of Section \ref{sec:lkdnhist}, respectively,
\enumrom
\item they have the same (in distribution) unlabelled graphs (Theorem \ref{thm:neut}), and
\item there is a permutation coupling $\sigma$ (see Definition \ref{def:coup}) of the two models with respect to which, conditioned on the edges in the forward neutral model up to generation $n$, the permutation $\sigma_n$ between vertices in generation $n$ is uniformly distributed.
\enumend

Property (i) is obtained by showing that a scrambled forward or backward neutral is  \emph{completely neutral} in the following sense.

\begin{definition}\label{def:unif}
$\E$ is \emph{completely neutral} if $(\E_n)_{n\ge 0}$ are independent and each $\E_n$ is exchangeable, i.e., $\E_n\eqd \{(\sigma(v),\sigma'(w))\colon (v,w) \in \E_n\}$ for every $\sigma \in \P_n$ and $\sigma'\in \P_{n+1}$.
\end{definition}

For a model based on a fixed $(k_n)$, exchangeability fixes the distribution of $\E_n$ (see the discussion above the proof of Theorem \ref{thm:neut} in Section \ref{sec:neut}), and the corresponding completely neutral model is unique in distribution. In particular, the following result implies property (i) above.

\begin{theorem}[Scrambling]\label{thm:neut}
Let $(V,\E)$ be a forward or backward neutral model with a fixed $V$ and let $\sigma \in \P$ be uniformly distributed and independent of $\E$. Then,
\enumar
\item $\sigma(\E)$ is completely neutral. 
\item  Moreover, 
\enumalph
\item if $(V,\E)$ is forward neutral then $(\sigma(\E_m))_{m\ge n}$ is independent of $(\sigma_m)_{m\le n}$, and
\item if $(V,\E)$ is backward neutral then $(\sigma(\E_m))_{m<n}$ is independent of $(\sigma_m)_{m\ge n}$.
\enumend
\enumend
\end{theorem}

Statements 2(a)-(b) may appear asymmetrical, but remember that $\E_n \subset V_n\times V_{n+1}$.\\

The following result won't be needed in the sequel, but it expresses nicely the statement that everyone in a forward neutral model has an equal chance of success. Define the descendant subtree $\T(v)$ to be the subgraph of $(V,\E)$ induced by the descendants of $v$, modulo isomorphism of directed graphs.

\begin{corollary}\label{cor:fwdtrees}
Let $(V,\E)$ be a forward neutral model. For each $n$, $(\T(v))_{v \in V_n}$ is exchangeable.
\end{corollary}

\begin{proof}
By Theorem \ref{thm:neut}, the distribution of $(\sigma(\E_m))_{m\ge n}$ conditioned on $\sigma_n=\alpha$ does not depend on $\alpha$, which means the distribution of $(\T(\alpha(v))_{v \in V_n})$ is the same for every $\alpha \in \P_n$.
\end{proof}

Theorem \ref{thm:neut} can be used to couple forward and backward neutral models, but does not immediately yield property (ii). Combining 1,2(b) of Theorem \ref{thm:neut} with Lemma \ref{lem:fwdcoup} we obtain the following result, which is property (ii) above (proof in Section \ref{sec:neut}).

\begin{theorem}[Forward-backward coupling]\label{thm:fwdback}
Let $(V,\E)$ be forward neutral and $(V,\E')$ be backward neutral, based on the same sequence of litter sizes $(k_n)$. There is a permutation coupling $\E'=\alpha(\E)$ such that for each $n$, $\alpha_n$ is uniformly distributed and independent of $(\E_m)_{m<n}$.
\end{theorem}
 
It is important to note that the random variables $(\alpha_n)$ are generally not independent.\\

By taking $(V,\E')$ in Theorem \ref{thm:fwdback} to be the ordered lookdown, we recover the following result, that generalizes Theorem 2.3 of \cite{donktz}. Note that models III \& IV from their paper are, respectively, a forward neutral model, and the ordered lookdown (which is backward neutral), as discussed in Sections \ref{sec:fwdneut}-\ref{sec:bwdneut} above.

\begin{corollary}[Lookdown coupling]\label{cor:lkdn}
Let $(V,\E)$ be a forward neutral model and $(V,\E')$ be the ordered lookdown representation given in Section \ref{sec:bwdneut}, both based on the same $(k_n)$. There is a permutation coupling $\E'=\alpha(\E)$ such that for each $n$, $\alpha_n$ is uniformly distributed and independent of $(\E_m)_{m<n}$.
\end{corollary}

\subsection{Size-biased ordering in the ordered lookdown}\label{sec:sbool}

As mentioned in the introduction, the ordered lookdown, defined precisely in Section \ref{sec:bwdneut}, arranges the individuals in each generation $n$ in size-biased order of the frequency of their descendants in each generation $m>n$, as described in Proposition \ref{prop:lkdnsel} below. The size-biased ordering property is well understood in the theory of exchangeable partitions of $\N$, as discussed in Section 1.2 in \cite{bercoal}, and to adapt it to this context requires essentially two ingredients:
\enumrom
\item the conditionally uniform property of the permutation $\sigma$ that defines the lookdown coupling (see Corollary \ref{cor:lkdn}) and
\item the fact that the ordered lookdown assigns, not only -- by definition -- parents to children, but in any pair of generations, ancestors to descendants, in increasing order of the least-ranked descendant (see Lemma \ref{lem:ordle}).
\enumend

The link between these properties and the desired result is then obtained by the equivalence of Algorithms 1-2 described below. We begin with some definitions.
\begin{definition}[Size-biasing]\label{def:sbrv}
Let $X$ be a random variable on $\R_+$ with $0<E[X]<\infty$. Then $X^*$ has the \emph{size-biased} distribution of $X$ if for $A\subset \R_+$,
\begin{align}
\pr(X^* \in A) = \dfrac{\ev[X \1(X \in A)]}{\ev[X]} 
\end{align}

\end{definition}

For a fairly recent review of size-biasing, see \cite{sbsurv}.

\begin{definition}[Size-biased sample]\label{def:sb}\,
\enumalph
\item Let $(x_1,\dots,x_\ell)\subset \R_+^\ell$ be a finite list with $\sum_i x_i>0$. \\ Then $x_I$ is a size-biased sample if $\pr(I=i) = x_i / \sum_j x_j$.

\item $B$ is a size-biased sample from a partition $P$ of a finite set $A$ if, for each $C\in P$,
$$\pr(B=C) = \frac{|C|}{|A|}.$$
\enumend

\end{definition}

If $P$ is a partition of a finite set $A$, then taking $u$ uniformly distributed on $A$, the block $B(u)$ containing $u$ is a size-biased sample from $P$, since
$$\pr(B(u)=B) = \pr(u\in B) = |B|/|A|.$$

Taking size-biased samples without replacement gives the size-biased ordering.

\begin{definition}[Size-biased ordering]\label{def:sbo}\,
\enumrom
\item $(y_1,\dots,y_\ell)$ is a size-biased ordering of a list $(x_1,\dots,x_\ell)$ if $y_i=x_{\sigma(i)}$ for some $\sigma$ for which, for all $k\le \ell$ and $j_1,\dots,j_k$ for which $x_{j_i}>0$ for all $i\le k$, \[\pr(\sigma(i)=j_i \ \forall i\le k) = \prod_{i=1}^k \frac{x_{j_i}}{\sum_{m=i}^\ell x_{j_m}}.\]

\item $(B_1,\dots,B_\ell)$ is a size-biased ordering of a partition $P$ if for each $i$, conditioned on $B_1,\dots,B_{i-1}$, $B_i$ is a size-biased sample from $P \setminus \{B_j\}_{j<i}$.

\enumend

$(y_i)_{i=1}^\ell$, resp.~$(B_i)_{i=1}^\ell$ is in size-biased order if it is a size-biased ordering of some $(x_i)_{i=1}^\ell$, resp.~$P$.
\end{definition}
The definition for lists accounts for the possibility that $x_i=0$ for one or more $i$, since this occurs for frequencies in genealogy models when a lineage goes extinct; note that when $x_i=0$ for one or more $i$, the distribution of $(y_i)_{i=1}^\ell$ is unique but the distribution of $\sigma$ is not. Also note that the definition for lists and partitions is consistent: if we choose an initial ordering of $P$ as $B_1,\dots,B_\ell$, then let $x_i=|B_i|$ and take $\sigma$ as above, then $B_{\sigma(1)},\dots,B_{\sigma(\ell)}$ is a size-biased ordering of $P$.\\

Size-biased ordering of the blocks of a partition $P$ generalizes uniform selection without replacement, which is the special case of singleton sets $P=\{\{a\}\colon a \in A\}$. 
Size-biased ordering is, essentially, a statistically appealing way of listing a set of numbers in roughly decreasing order, or a collection of sets in roughly decreasing size. It's also the result of a very natural sampling procedure, expressed in the following two equivalent ways.

\begin{algorithm}\label{alg:sbo1}
To place the blocks of a partition in size-biased order, sample uniformly without replacement from the underlying set, then list the blocks in the order they are discovered. In other words, given $A$ and $P$, and $B(u), u\in A$ as above, let $n=|A|$ and $u:\{1,\dots,n\}\to A$ be a uniform random bijection, then let $B_1=B(u(1))$ and for $i>1$, $I_i:=\min\{i\colon u(i) \notin \bigcup_{j<i}B_j\}$ and $B_i = B(u(I_i))$. The property quoted in Definition \ref{def:sbo} follows from the fact that, conditioned on $u(1),\dots,u(I_{i-1})$, $u(I_i)$ is uniform on $A \setminus \bigcup_{j<i}B_j$.
\end{algorithm}

\begin{algorithm}\label{alg:sbo2}
An equivalent formulation to Algorithm \ref{alg:sbo1} is to (i) scramble $A$ and map it onto an ordered set, then (ii) order the scrambled blocks by least element. Let $\sigma:A\to\{1,\dots,n\}$ be a uniform random bijection and order the blocks in increasing order of $\min \sigma(B)$ as $(B_i)_{i=1}^\ell$. Then $u(i):=\sigma^{-1}(i)$ and $I_i:=\min \sigma(B_i)$ are exactly as in Algorithm \ref{alg:sbo1}, i.e., $(B_i)_{i=1}^\ell$ is a size-biased ordering of $P$.
\end{algorithm}

Say that $(B_1,\dots,B_\ell)$ sorts $P$ by least element if $P=\{B_i\}_{i=1}^\ell$ and $i\mapsto \min B_i$ is increasing.\\


Let $D_{n,m}(i)=(D_m((n,i))$, $i=1,\dots,X_n$ denote the vector of descendants, in generation $m$, of vertices in $V_n$ and let $\ell_{n,m}=|\Xi_{n,m}|$ be the number of individuals in generation $n$ that have at least one descendant in generation $m$, with $\ell_n:=\ell_{n,n+1}$. Give $V_n$ the order $(n,i)<(n,j)$ if $i<j$.

\begin{lemma}\label{lem:ordle}
For the ordered lookdown and every $n<m$, $(D_{n,m}(i))_{i=1}^{\ell_{n,m}}$ sorts $\Xi_{n,m}$ by least element.
\end{lemma}

\begin{proof}
Define the \emph{min path} of $v\in V_n$ by $\gma_m(v)=\min D_m(v)$, with the convention $\min\emptyset=\infty$, and for each $n,i$ let $\Gma_n(i)=\min O((n,i))$ denote the lowest-ranked offspring of $(n,i)$ (or $\infty$ if it has none). By definition of the lookdown, $i_1<i_2\le\ell_n$ implies $\Gma_n(i_1)<\Gma_n(i_2)$, from which it follows easily that
\[\gma_m((n,i)) = \Gma_{m-1}\circ \dots \circ \Gma_n(i).\]
In particular, if $i_1<i_2$ and $\gma_m((n,i_2))<\infty$ then $\gma_m((n,i_1))<\gma_m((n,i_2))$. Since there are $\ell_{n,m}$ values of $i$ for which $\gma_m((n,i))<\infty$, they must be $1,\dots,\ell_{n,m}$, and combining with the previous observation, the claim follows.
\end{proof}

For a vertex $v$ define the extinction time of the clade of $v$ by $\tau(v)=\inf\{m\colon D_m(v)=\emptyset\}$. The ordered lookdown arranges vertices in each generation in non-increasing order of $\tau(v)$, as follows.

\begin{corollary}\label{cor:tauord}
For $v \in V$ let $\tau(v)= \inf\{m\colon D_m(v)=\emptyset\}$. For the ordered lookdown, for each $n$, $i\mapsto \tau((n,i))$ is non-increasing.
\end{corollary}

\begin{proof}
Define the min path $\gma_m(v)$ as in the proof of Lemma \ref{lem:ordle}, noting that $i\mapsto \gma_m((n,i))$ is non-decreasing for each $n,m$. By definition, $\tau(v)=1+M(v)$ where for $v\in V_n$, $M(v)=\inf\{m\ge n\colon \gma_m(v)>\ell_m\}$. If $j>i$ and $m=M((n,i))$ then $\gma_m((n,j))\ge \gma_m((n,i))>\ell_m$ so $m\ge M((n,j))$ and $\tau((n,j))=1+M((n,j))\le 1+M((n,i))=\tau((n,i))$, as desired.
\end{proof}

Combining Lemma \ref{lem:ordle} with the lookdown coupling and recognizing Algorithm \ref{alg:sbo2}, we find that the ordered lookdown arranges vertices in size-biased order of the number of their descendants, at any later generation. Write $X_{n,m}(i)=X_m((n,i))$ and let $\ell_{n,m}=|\Xi_{n,m}|$ be the number of individuals in generation $n$ that have at least one descendant in generation $m$ and let $\ell_n^*=\#\{v \in V_n\colon x(v)>0\}$.

\begin{proposition}[Size-biased ordering in the ordered lookdown]\label{prop:lkdnsel}
Let $\E$ denote the ordered lookdown and $\sigma\in \P$ uniform and independent of $\E$.
\enumrom
\item For each $n<m$, conditioned on $(\sigma(\E_k))_{k<m}$, $X_{n,m}$ is in size-biased order.
\item For each $n$, conditioned on $\sigma(\E)$, $(x((n,i)))_{i=1}^{X_n}$ is in size-biased order.
\enumend
\end{proposition}

The proof is given in Section \ref{sec:sbo}.

\subsection{Identification probability}\label{sec:ident_def}

In Section 7.3, Question 2 of \cite{sbsurv} the authors ask whether, given the spinal decomposition of a size-biased Galton-Watson tree (see Section \ref{sec:sd} below for the definition), but not being told which particles form the spine, whether it is possible to identify the spine. To formulate this problem, in Section 7.4 they consider the problem of detecting a random variable sampled from a distribution $Q$, when mixed in with $k-1$ i.i.d.~random variables with distribution $P$, all taking values in a set $S$. Mathematically, they formulate the problem as follows: $Y_1$ has distribution $Q$, $Y_2,\dots,Y_k$ have distribution $P$, and $X_i=Y_{\pi(i)}$ where $\pi$ is a uniform random permutation of $[k]:=\{1,\dots,k\}$, and the problem is to determine $I:=\pi(1)$ by examining $\mathbf{X}:=(X_1,\dots,X_k)$. They define a \emph{selection procedure} to be any function $f:S^k \to [k]$, and its \emph{score} is the probability of correct identification, that is, $s(f):=\pr(f(\mathbf X)=I)$, where $\pr$ is the joint measure of $(Y_1,\dots,Y_k)$ and $\pi$. They observe that the best selection procedures choose $I$ arbitrarily from those indices that maximize the likelihood ratio $r(x)=dQ/dP(x)$, then define the identification probability to be $s(f_0)$ where $f_0$ is any such procedure, or equivalently, $\ev[\max_i r(X_i)]$.\\

We will use an analogous definition, adapted to our context. In our case, the object to be identified is the image of a vertex $v$ under the mapping that takes a labelled graph to its unlabelled counterpart. If the vertex set $V$ is fixed, then in the terminology of Section \ref{sec:gengrph} this amounts to scrambling $(V,\E)$ by applying a uniform random permutation $\sigma$ independent of $\E$, and to attempt to determine $\sigma(v)$ by examining only $\sigma(\E)$. In place of the likelihood ratio, we use the conditional probability of correctly identifying $\sigma(v)$ given $\sigma(\E)$.\\

Let $\pr$ be the joint distribution of $\E$ and $\sigma$ and suppose $v\in V_n$. It's easy to check that $\pr$ can be generated by an i.i.d.~sequence of $\unif[0,1]$ random variables, so it lives on a standard Borel space. This implies existence of the regular conditional distribution of $\sigma(v)$ given $\sigma(\E)$ (see for example Section 4.1.3 in \cite{PTE}). Denote this by
\[\mu(v,w)=\pr(\sigma(v)=w\mid \sigma(\E)).\] Given $\sigma(\E)$, our best guess of $\sigma(v)$ is obtained by taking any element of \[\arg\max\{\mu(v,w)\colon w \in V_n\},\]
and the probability that we guessed correctly is $\max_{w\in V_n}\mu(v,w)$. Averaging over the value of $\sigma(\E)$, we can define the identification probability of $v$ with respect to $\E$ by
\begin{align}\label{eq:ip}
\rho(v):= \ev[ \,\max_{w \in V_n} \mu(v,w)\,].
\end{align}
A similar definition can be used for a list $(v_i)_{i=1}^k$, $k\le \infty$ or a set $U\subset V$ of vertices: for $(w_i)_{i=1}^k \in V^k$ and $W\subset V$ define
\begin{align}\label{eq:muVW}
&\mu((v_i),(w_i))=\pr(\sigma(v_i)=w_i \ \forall i\le k\mid \sigma(\E)), \nonumber \\
&\mu(U,W)=\pr(\sigma(U)=W \mid \sigma(\E))
\end{align}
and define the identification probability by
\begin{align}\label{eq:ip2}
&\rho((v_i)):= \ev[ \,\max_{(w_i) \in V^k} \mu((v_i),(w_i))\,],\nonumber \\
&\rho(U):= \ev[ \,\max_{W\subset V} \mu(U,W)\,].
\end{align}
\subsection{The spinal decomposition}\label{sec:sd}

Following \cite{LPP}, the spinal decomposition (SD) of a Galton-Watson tree is a probability measure on trees with a distinguished path. It exists if the mean number of offspring $\mu:=\sum_k k\,p_k$ is finite. We'll represent it as $(V,\E,\gma)$ where $\gma:\N \to V$, the spine, has $(\gma_n,\gma_{n+1})\in \E_n$ for every $n$. As in Section \ref{sec:basex} let $(K_{n,i})$ be i.i.d.~with $P(K_{n,i}=k)=p_k$ and let $(K_{n,i}^*)$ be i.i.d.~with $\pr(K_{n,i}^*=k)=k\,p_k \,  / \mu$, so that $K_{n,i}^*$ has the size-biased distribution of $K_{n,i}$, as in Definition \ref{def:sbrv}. Let $X_0=1$ and $\gma_0=1$, and let $X_1=K_{0,1}^*$, $\gma_1$ be uniform random on $\{1,\dots,X_1\}$ and $\E_0=\{((1,1),(2,i)\colon i\in\{1,\dots,X_1\}\}$. Then for each $n$, let
\begin{align*}
&X_{n+1} = K_{n,\gamma_n}^* + \sum_{1\le i \le X_n, \ i \ne \gamma_n} K_{n,i}, \\
&c_{n,i}  = \1(i\ge \gamma_n) K_{n,i}^* + \sum_{1 \le j \le i,  \ j \ne \gamma_n} K_{n,j},
\end{align*}
define $\E_n$ as in \eqref{eq:Endet} and let $\gma_{n+1}$ be uniform random on $\{c_{n,\gma_n-1},\dots, c_{n,\gma_n}\}$. Continuing in this way defines the SD. Note that litter sizes along the spine are size-biased, the path of the spine is chosen uniformly at random from its children at each step, and non-spine litter sizes have the same distribution as in the corresponding GW tree.\\

The usefulness of the above construction is its connection to the size-biased GW tree. For a GW tree, the process of rescaled population sizes $(W_n)$ on $\R_+$ defined by $W_n=X_n/\mu^n$ is a martingale, so the measure of the GW tree biased on $(W_n)$, called the size-biased GW tree, can be defined. As explained in Section 2 of \cite{LPP}, for each $n$, the distribution of the SD restricted to generations $\{0,\dots,n\}$ is equal to that of the graph obtained from a size-biased GW tree $(V,\E)$ by letting $\gma_n$ be a uniform random vertex in $V_n$ and defining $(\gma_m)_{m\le n}$ to be the unique path from $V_0=\{(0,1)\}$ to $\gma_n$. The SD is used in \cite{LPP} to prove, among other things, the Kesten-Stigum theorem, that says that $W:=\lim_{n\to\infty} W_n$ is positive on the event of survival if $\sum_k k\log_+(k)p_k<\infty$ and $W=0$ almost surely otherwise.\\

The connection of the SD to the ordered lookdown and the properties of takeover, identifiability and fixation are discussed in Section \ref{sec:subgw}.

\subsection{Coalescent time scale}\label{sec:coalts}

Define the \emph{coalescent time scale} $(t_n)$ by $t_n = \sum_{m=0}^{n-1}s_m$, where $s_n$ is the probability that a uniform random pair of distinct vertices $v,w \in V_{n+1}$ have the same parent in $V_n$. If the increments $(s_n)$ are small, then $(t_n)$ is the time scale on which, looking backwards, the waiting time until the most recent common ancestor of $v$ and $w$ is approximately exponentially distributed with mean $1$.

\section{Main results}\label{sec:main}

\subsection{Takeover}\label{sec:tkover}

As described in the introduction, we say that takeover occurs if there is single lineage for which the proportion of the population in generation $m$ that can trace their ancestry to that lineage tends to $1$ as $m\to\infty$. In terms of generational graphs and limiting frequencies this can be expressed as follows: there exist $(v_n)_{n\ge 0}$ with $v_n\in V_n$ and $v_n=a(v_{n+1})$ for each $n$, such that $x(v_n)=1$ for all $n$. Of course, if $x(v_n)=1$ then $x(w)=0$ for $w\in V_n\setminus \{v\}$ so the lineage is unique, when it exists. Note that takeover does not necessarily imply fixation, i.e., that every other lineage goes extinct; a notable counterexample is  a size-biased GW tree with $\sum_k k\log_+(k)p_k=\infty$, see Section \ref{sec:subgw}.\\

The following result says that the following are equivalent:
\enumrom
\item takeover occurs,
\item infinite time passes on the coalescent time scale.
\enumend
\begin{theorem}[Takeover]\label{thm:tkover}
Let $(V,\E)$ be the ordered lookdown based on some fixed $(k_n)$ and let $t_\infty = \lim_{n\to\infty} t_n$ where $(t_n)$ is the coalescent time scale. Let $u_n=(n,1)$ denote the vertices on the base path.
\enumrom
\item if $t_\infty=\infty$ then $x(u_n)=1$ almost surely for all $n$, and
\item if $t_\infty<\infty$ then $\max_{v \in V_n}x(v)\to 0$ in probability as $n\to\infty$.
\enumend
\end{theorem}

The basic idea is as follows: the coalescent time scale is the time scale on which, as $m$ increases, the frequencies $(x_m(v))_{v \in V_n}$ diffuse towards one of the absorbing states, namely, $(1,0,\dots,0)$ or a permutation thereof. In the lookdown, each vertex on the base path has a built-in reproductive advantage that becomes more pronounced, the more time passes on the coalescent time scale $(t_n)$. In particular, since vertices are arranged in size-biased order of frequency, if there is a vertex whose frequency tends to $1$ then it must be $u_n$. If $t_\infty=\infty$ this process of diffusion has sufficient time to reach the absorbing state, while if $t_\infty<\infty$ then $\sum_{m\ge n}s_m \to 0$ as $n\to\infty$ so for large $n$, there is little time between generations $n$ and $\infty$ for the descendants of any particle to develop much of an advantage.\\

The takeover result is intuitive once the ingredients are understood, and the proof is fairly painless once we're equipped with the ordered lookdown and the coalescent time scale. Though intuitive, the result draws an elegant and natural connection between the coalescent time scale and the (martingale) frequencies in a forward neutral model. The precise formulae that achieve this connection are given by the following lemmas, which are proved in Section \ref{sec:dom}. Let $p_{n,m}$ denote the probability that a uniform random pair of distinct vertices in generation $m$ have a common ancestor in generation $n$. The first result connects $p_{n,m}$ to the coalescent time scale and is not hard to prove.

\begin{lemma}\label{lem:coalprob}
For every $n<m$,
\begin{align}\label{eq:coalprob}
p_{n,m} = 1-\prod_{j=n}^{m-1} (1-s_j).
\end{align}
In particular, $m\mapsto p_{n,m}$ is non-decreasing and \\$p_{n,\infty}:= \lim_{m\to\infty} p_{n,m}=1$ iff $s_j=1$ for some $j\ge n$ or $\sum_j s_j=\infty$.
\end{lemma}

The second result connects $p_{n,\infty}$ to the limiting frequencies on the base path of the lookdown. We rule out the case in which eventually $k_n=(1,\dots,1)$; it messes up the formula and is a trivial case since it means that eventually there is no reproduction (aside from 1:1 replacement) and no death.
\begin{lemma}\label{lem:pnm_freq}
Let $u_n=(n,1)$ denote vertices on the base path of the lookdown and assume $k_n\ne (1,\dots,1)$ infinitely often. For each $n$,
\begin{align}\label{eq:esbasymfreq}
\ev[x(u_n)] = \ev[\sum_{v\in V_n}x(v)^2] = p_{n,\infty}.
\end{align}
\end{lemma}
The first equality expresses the mean of a size-biased sample in terms of the second moment of element sizes, while the second expresses the probability that an independent uniform random pair of points fall in the same element of an interval partition of $[0,1]$, see for example (4) in \cite{bercoal}.

\subsubsection{Asynchronous models}\label{sec:asyn}

To get an idea of the conditions that lead to convergence or divergence of $t_\infty$ it helps to consider the following special case.

\begin{definition}\label{def:asyn}
A model is \emph{asynchronous} if only one individual dies or gives birth in each generation, i.e., if there is a sequence of non-negative integers $(b_n)_{n<\tau-1}$ such that $X_{n+1}=X_n+b_n-1$ and $k_n = (b_n,1,\dots,1)$ for each $n$.
\end{definition}

For an asynchronous model, if $X_{n+1}=1$ then $s_n=1$, and otherwise,
\begin{align}\label{eq:asynsn}
s_n= \displaystyle\frac{\binom{b_n}{2}}{\binom{X_{n+1}}{2}} = \frac{b_n(b_n-1)}{X_{n+1}(X_{n+1}-1)}.
\end{align}

Note that $s_n=1$ iff $X_n=1$. Since $b_n-1=X_{n+1}-X_n$ is the first difference of the sequence $(X_n)$, $s_n$ is comparable to
\[\1(b_n\ge 2) \frac{(X_{n+1}-X_n)^2}{X_{n+1}^2}\]
which, with a bit of effort, we can show is comparable to the quadratic variation process of $\log(X_n)$, counted during increases in $X_n$. This leads to the following tidy characterization.

\begin{theorem}\label{thm:tovar}
Let $\Delta_n = \max(0,\log(X_{n+1}/X_n))$ and suppose that $X_n>1$ eventually. \\For an asynchronous model,
\enumrom
\item If $X_{n+1}\ge 2X_n$ then $s_n\ge 1/4$.
\item If $X_{n+1} \le 2X_n$ then $1/4\le s_n/\Delta_n^2 \le 4(\log 2)^2$.
\enumend
In particular, $t_\infty=\infty$ if and only if
\[\sum_n \Delta_n^2=\infty.\]
\end{theorem}

If we let $Y_n=\log X_n$ then $\Delta_n= (Y_{n+1}-Y_n)_+$, so for a given $Y_0$ and $Y_n$ (and ignoring that $X_n$ must be integer-valued) the values of $Y_1,\dots,Y_{n-1}$ that minimize $\sum_{m=0}^{n-1}\Delta_m^2$ are equally spaced, i.e., $Y_m=(Y_n-Y_0)m/n$ (see the proof of Theorem \ref{thm:tovar}). This means that for a given asymptotic growth profile, i.e., a function $f(n)$ such that $Y_n\sim f(n)$, coalescence will be slowest when the per-generation growth rate is as near to constant as possible.\\

With the above notion in mind, it's interesting to note that when growth is either sufficiently slow (with enough birth and death) or sufficiently fast, then takeover cannot be avoided, as expressed by the following result. 

\begin{theorem}\label{thm:asynfs}
For an asynchronous model, if $\tau=\infty$ and either
\enumrom
\item $X_n = O(\sqrt{n})$ and $\limsup_n(1/n)\#\{m\le n\colon b_m=1\}<1$ or
\item $\limsup_n \log(X_n)/\sqrt{n}>0$,
\enumend
then $t_\infty=\infty$.
\end{theorem}

This result agrees qualitatively with the observations for size-biased Galton-Watson trees in section \ref{sec:subgw}, namely that if $\mu:=\sum_k kp_k\le 1$ then fixation, and thus takeover, occurs, while if $\mu>1$ then takeover occurs iff $\sum_k k\log_+(k)p_k=\infty$. In other words, as the growth rate increases there is a transition from takeover, to no takeover, then back to takeover, in agreement with the results of Theorem \ref{thm:asynfs}. We can go a bit further in making this connection precise. To obtain the genealogy model that corresponds to the jump chain of a continuous time branching process $(Z_t)$ with offspring distribution $(p_k)$, i.e., with
\[Z \to Z+k-1 \ \text{at rate} \ Zp_k,\]
take the forward neutral asynchronous model in which $(b_n)_{n\ge 0}$ are independent random variables with $\pr(b_n=k)=p_k$. Arguing informally, if $Z$ is critical, i.e., $\sum_k kp_k=1$, and in addition $\sum_k k^2p_k<\infty$, then from Section 4 in \cite{LPP}, then under the size-biased measure $Z_t/t$ converges in distribution to a positive limit $W$, so the number of events in $[0,t]$ is of order $2\int_0^t Z_sds \sim Wt^2$ and the jump times $(t_n)$ of $Z$ have $t_n \sim \sqrt{n/W}$, so the jump chain $Y_n:=Z_{t_n}$ has $Y_n \sim c\sqrt{n}$. In other words, population growth of a size-biased critical branching process is just slow enough that Theorem \ref{thm:asynfs} (i) applies. This does not, by itself, provide new information, since a size-biased critical GW tree has a unique infinite path (see Section \ref{sec:subgw}), i.e., fixation holds, which implies takeover. However, it suggests that this case lies on the edge of takeover, and that slightly faster, steady population growth will lead to multiple competing lineages with no clear winner, as in the case $\mu>1$ and $\sum_k k\log_+(k)p_k<\infty$. Case (ii) in Theorem \ref{thm:asynfs} can of course also be achieved in the above context (i.i.d.~$(b_n)$) by making $(p_k)$ sufficiently heavy-tailed; we leave it to the interested reader to determine whether or not the threshold in case (ii) is close to the takeover threshold $\sum_k k\log_+(k)p_k=\infty$.

\subsection{Fixation}\label{sec:subfix}

Recall that fixation occurs if in each generation, at most one vertex has infinitely many descendants. The following result gives a sufficient condition for fixation. First, note that in general,
\begin{align}\label{eq:snform}
s_n = \sum_{i=1}^{X_n} \frac{\binom{k_n(i)}{2}}{\binom{X_{n+1}}{2}} = \sum_{i=1}^{X_n}\frac{k_n(i)(k_n(i)-1)}{X_{n+1}(X_{n+1}-1)}.
\end{align}
For each $n$ let $L_n=\#\{i\colon k_n(i)\ge 2\}$ be the number of multi-child litters in generation $n$, then let
\begin{align}\label{eq:sntrunc}
s_n^o = \frac{2 L_n}{X_{n+1}(X_{n+1}-1)}
\end{align}
and define the \emph{truncated coalescent time scale} $(t_n^o)$ by $t_n^o=\sum_{m=0}^{n-1} s_m^o$, with $t_\infty^o:=\lim_{n\to\infty} t_n^o$. Note that $s_n^o \le s_n$ with equality iff $\max_ik_n(i)\le 2$, so $t_n^o \le t_n$, with equality iff $\max_{n,i} k_n(i) \le 2$.

\begin{theorem}\label{thm:fix}
Let $(V,\E)$ be a neutral model based on $\tau,(X_n)$ and $(k_n)$, with $\tau=\infty$ and let $(t_n^o)$ be the truncated coalescent time scale. If $t_\infty^o=\infty$ then fixation occurs almost surely.
\end{theorem}

Even working with the truncated time scale, the proof of this result has some technical challenges. As explained in Section \ref{sec:fix}, our method proceeds by finding lower bounds on ranks of descendants of particles ordered in the lookdown, with the target estimates determined by examining the Moran model, then with increasing effort, proving these estimates hold in the asynchronous case, then in the general case; to reach the general case, several technical lemmas are needed.\\

It's clear that fixation implies takeover, since if $\tau(v)=\infty$ and $\tau(w)<\infty$ for all $w\in V_n\setminus \{v\}$ then $x_m(v)=1$ for all $m>\max\{\tau(w)\colon w \in V_n\setminus \{v\}\}$. To achieve the converse implication it suffices, by Theorem \ref{thm:fix}, that $s_n=O(s_n^o)$; defining the empirical second moment of multi-child litters by
\[m_{n,2} = \frac{1}{L_n}\sum_{i=1}^{X_n}k_n(i)^2 \1(k_n(i) \ge 2),\]
if $(m_{n,2})$ is bounded then $s_n \le (\sup_n m_{n,2}/2)s_n^o$, and combining Theorems \ref{thm:tkover} and \ref{thm:fix} we obtain the following result.

\begin{corollary}
Let $(V,\E)$ be a neutral model based on $(k_n)$, with $\tau=\infty$, and $m_{n,2}$ as above. If $\sup_n m_{n,2}<\infty$ then the following are equivalent:
\enumrom
\item $t_\infty=\infty$, where $(t_n)$ is the coalescent time scale.
\item fixation occurs almost surely.
\item fixation occurs with positive probability.
\enumend
\end{corollary}

It's easy to show the conclusion may not hold when $m_{n,2}$ is unbounded: for example, consider the asynchronous case and let $X_0=1$ and $b_n=2^n$, so that $X_n=2^n$. Since there are no deaths, all paths are infinite, and from \eqref{eq:asynsn}, $s_n \to 1/4$ as $n\to\infty$, implying $t_\infty=\infty$. So, while we do not claim the conditions for fixation given in Theorem \ref{thm:fix} are sharp, they are reflective of the fact that, in general, takeover does not imply fixation.

\subsection{Identifiability}\label{sec:ident}

The problem of identification (see Section \ref{sec:ident_def}) is to determine the image of a particular vertex (or a list/set of vertices) after scrambling the graph, by examining the scrambled graph. We study this problem for the ordered lookdown, giving formulae for some identification probabilities as well as conditions under which those probabilities are equal to 1. We will mostly focus on the problem of identifying the vertices in $V_n$ for a fixed $n$.\\

Let $\E$ denote the lookdown and $\sigma$ be uniform, so that $\E':=\sigma(\E)$ is the scrambled graph. Let $\alpha=\sigma^{-1}$. Since $\sigma(v)=w$ iff $v=\alpha(w)$, the problem of identification amounts to understanding the conditional distribution of $\alpha$ given $\E'$. For this purpose it's useful to define the \emph{rank} of $w\in V_n$, denoted $r(w)$, to be the unique $i$ such that $(n,i)=\alpha(w)$. In other words, for a vertex $w$ in the scrambled graph, $r(w)$ is its position in the lookdown. Since $\alpha$ and $r$ encode the same information, it suffices to study $r$.\\

There are two ways of obtaining information about the ranks of vertices in a given generation, namely, by comparing the extinction times of their clades, and the frequency of their descendants in later generations, as described by Corollary \ref{cor:tauord} and Proposition \ref{prop:lkdnsel} respectively.\\

\nid\tbf{Extinction times. }As in Corollary \ref{cor:tauord} let $\tau,\tau'$ denote extinction time with respect to $\E,\E'$ respectively, so that $\tau'(w)=\tau((n,r(w)))$ for $w\in V_n$. Corollary \ref{cor:tauord} tells us that if the clades of vertices $w,u\in V_n$ have different extinction times with respect to $\E'$, say $\tau'(w)<\tau'(u)$, then their ranks have the opposite order, i.e., $r(w)>r(u)$, giving the following result.

\begin{lemma}\label{lem:distext}
With notation as above, order the set $\{\tau(v)\colon v \in V_n\}$ as $\infty\ge \tau_1>\dots>\tau_q $ and let $W_n^{(i)}=\{w\in V_n\colon \tau'(w)=\tau_i\}$. Let $c_n(0)=0$, $c_n(i)=\sum_{j\le i}|W_n^{(i)}|$ and let $V_n^{(i)}=\{c_n(i-1)+1,\dots,c_n(i)\}$. Then $r(W_n^{(i)})=V_n^{(i)}$ for each $i$.
\end{lemma}

In particular, when extinction times are distinct we can recover all ranks with perfect accuracy.

\begin{definition}\label{def:asyndth}
A model has \emph {asynchronous deaths} if at most one individual dies in each generation, i.e., if $\#\{i\colon k_n(i)=0\} \le 1$ for each $n$.
\end{definition}

\begin{theorem}\label{thm:ipasyn}
Let $(V,\E)$ be the ordered lookdown, based on a fixed $(k_n)$, and suppose it has asynchronous deaths and at most one infinite lineage (i.e., fixation holds). Let $(v_i)$ be an enumeration of $V$ and $\rho((v_i))$ as in \eqref{eq:ip2}. Then $\rho((v_i))=1$.
\end{theorem}

\begin{proof}
Let $\sigma\in\P$ be uniform and independent of $\E$ and let $\E'=\sigma(\E)$ denote the scrambled graph. Let $\alpha=\sigma^{-1}$. It suffices to show that $\alpha$ is determined by $\E'$, since then, for any enumeration $(w_i)$ of $V$,
\[\pr( \sigma(v_i)=w_i \ \forall i \mid \sigma(\E)) = \1(v_i=\alpha(w_i) \ \forall i)\]
so letting $(w_i)$ range over all possible enumerations of $V$,
\[\max_{(w_i)}\pr(\sigma(v_i)=w_i \ \forall i \mid \sigma(\E)) = 1\]
and from \eqref{eq:ip2}, $\rho((v_i)) = \ev[\, 1 \,]=1$. To show that $\alpha$ is determined by $\E'$, first note that asynchronous deaths imply extinction times are distinct, i.e., that for each $n$ the $W_n^{(i)}$ of Lemma \ref{lem:distext} are singletons $\{w_n(i)\}$, and $(w_n(i)\colon i=1,\dots,X_n, \ n\in \N)$ enumerate $V$. By Lemma \ref{lem:distext}, $r(w_n(i))=(n,i)$ a.s.~for all $n,i$. Now, $\alpha$ is determined by $r$, which is determined by $(w_n(i))$, which are determined by $\tau'(\cdot)$, which is determined by $\E'$, as we wanted to show.
\end{proof}

\nid\tbf{Relative frequency. }Proposition \ref{prop:lkdnsel} tells us that descendant frequencies in the lookdown follow a size-biased ordering. Let $x_m,x_m'$ denote frequency with respect to $\E,\E'$ respectively, so that $x_m'(w)=x_m((n,r(w)))$ for $w\in X_n$.

\begin{lemma}\label{lem:rcm}
If $(w_i)_{i=1}^{X_n}$ is an enumeration of $V_n$, $\ell \in\{1,\dots,X_n\}$ and $x_m'(w_i)>0$ (equivalently, $\tau'(w_i)>m$) for all $i\le \ell$ then
\begin{align}\label{eq:rcm}
p_m((w_i)_{i=1}^\ell):=\pr (r(w_i)=i \ \forall i\le \ell \mid (\E'_k)_{k<m}) = \prod_{i=1}^\ell \frac{x_m'(w_i)}{\sum_{j=i}^{X_n} x_m'(w_j)}.
\end{align}
With $v_i=(n,i)$ and $\mu(\cdots)$ as in \eqref{eq:muVW}, if $\tau'(w_i)=\infty$ for all $i\le \ell$ then $p_m((w_i)_{i=1}^\ell \to \mu((v_i)_{i=1}^\ell,(w_i)_{i=1}^\ell)$ a.s.~and in $L^1$ as $m\to\infty$.
\end{lemma}

\begin{proof}
By definition, the left-hand side of \eqref{eq:rcm} is equal to
\[\ev[\mu((v_i)_{i=1}^\ell,(w_i)_{i=1}^\ell) \mid (\E_k')_{k<m}]\]
so if $\tau'(w_i)=\infty$ for all $i\le \ell$, $(p_m((w_i)_{i=1}^\ell)_{m\ge n}$ is a UI martingale, and the result follows.
\end{proof}

\begin{theorem}\label{thm:iptauinf}

For each $n$ and $w,u\in V_n$ for which $\tau(w)=\tau(u)=\infty$ the limiting relative frequency
\begin{align}\label{eq:limrelfreq}
x(w,u)=\lim_{m\to\infty} \frac{x_m(w)}{x_m(u)} \in [0,\infty]
\end{align}
exists. Let $(w_i)_{i=1}^{X_n}$ be an ordering of $V_n$ for which $x(w_j,w_i)\le 1$ for all $j\ge i$. For any $\ell \in \{1,\dots,X_n\}$
\begin{align}\label{eq:tauinfip}
\rho(((n,i))_{i=1}^\ell) = \ev\left[ \prod_{i=1}^\ell \bigg(1+\sum_{j=i+1}^{X_n}x(w_j,w_i)\bigg)^{-1}\right].
\end{align}
In particular,
\begin{align}\label{eq:baseip}
\rho((n,1)) = \ev[\max_{v\in V_n}x(v)].
\end{align}
\end{theorem}
Aside from dealing with limits as $m\to\infty$, \eqref{eq:tauinfip} amounts to showing that the ordering with maximal probability in a size-biased ordering is the obvious choice, i.e., arranging entries in decreasing order.\\

Combining \eqref{eq:baseip} with Theorem \ref{thm:tkover} and noting that $\max_{v \in V_n}x(v) \ge x(u_n)$ we see that $\rho((n,1))=1 \ \forall n$ if $t_\infty=\infty$ and $\rho((n,1))\to 0$ as $n\to\infty$ if $t_\infty<\infty$. After showing that $\rho((n,1)_{n\ge 0})=\lim_{n\to\infty}\rho((n,1))$ (see Section \ref{sec:ip}) we obtain the following result; in words, if infinite time passes on the coalescent time scale then the base path can be identified from the unlabelled graph with complete certainty just by examining the edges, but if not, there is no chance of identifying the (entire) base path correctly. 

\begin{corollary}\label{cor:baseip}
For the ordered lookdown based on a fixed $(k_n)$,
\enumrom
\item If $t_\infty=\infty$ then $\rho((n,1)_{n\ge 0})=1$, and
\item if $t_\infty<\infty$ then $\rho((n,1)_{n\ge 0})=0$.
\enumend
\end{corollary}
The result is quite natural from the following perspective: since they come first in the size-biased ordering, litter sizes along the base path are statistically larger than those of other particles, so the base path can be picked out with certainty iff the effect of that advantage has enough time, on the coalescent time scale, to be visible. It's interesting to note that the base path need not be the unique infinite path in order to be identifiable, since we can also rely on frequency to find it.\\

We expect that extinction time and relative frequency suffice to fully describe the conditional distribution of $r$ given $\E'$, and that the following holds.

\begin{conjecture}\label{conj:ranks}
With $(\tau_i)_{i=1}^q, (W_n^{(i)})_{i=1}^q$ as in Lemma \ref{lem:distext}, conditioned on $\E'$,

\enumrom
\item the random vectors $\{(r(w))_{w \in W_n^{(i)}} \colon i \in \{1,\dots,q\}\}$ are independent,
\item for any ordering $(w_j)_{j=1}^k$ of $W_n^{(i)}$,
\[\pr(r(w_j)=j \ \forall j\le k \mid \E')=\prod_{j=1}^k \frac{x_{\tau_i-1}'(w_j)}{\sum_{a=j}^k x_{\tau_i-1}'(w_a)}\]
\enumend
\end{conjecture}

In words, we expect the ranks of vertices with different extinction times are independent, and that for each value of the extinction time, the corresponding ranks are determined by a size-biased ordering of their frequencies just prior to extinction. If Conjecture \ref{conj:ranks} holds, then in principle there is always a formula for the identification probability in terms of extinction time and relative frequency, though it will be a bit messy in general. We would, however, have the following straightforward generalization of Theorem \ref{thm:ipasyn}: all vertices in the ordered lookdown can be identified with probability 1 iff in each generation, (i) vertices whose clades go extinct in finite time have distinct extinction times, i.e., for all distinct $v,w \in V_n$ for which $\tau(v),\tau(w)<\infty$, $\tau(v)\ne \tau(w)$, and (ii) vertices whose clades never go extinct have relative frequency $0$ or $\infty$, i.e., for all distinct $v,w \in V_n$ for which $\tau(v)=\tau(w)=\infty$, $x(v,w) \in \{0,\infty\}$; the latter part makes use of \eqref{eq:tauinfip}. It would be interesting to know whether there exists a neutral model with multiple surviving lineages for which relative frequencies are all $0$ or $\infty$.

\subsection{The spinal decomposition}\label{sec:subgw}

For size-biased Galton-Watson (GW) trees, the question of identifiability of the spine is addressed in \cite{sbsurv}. As we show, the spinal decomposition can be obtained from the ordered lookdown, with the spine corresponding to the base path, so we also discuss connections with those results. The GW tree with offspring distribution $(p_k)$ is constructed in Section \ref{sec:basex} and is shown to be forward neutral in Section \ref{sec:fwdneut}.\\

\nid\tbf{Spinal decomposition. }Following \cite{LPP}, the spinal decomposition (SD) is a GW tree with a distinguished path, the spine, along which the distribution of litter sizes is size-biased, and it exists if the mean number of offspring $\mu:=\sum_k k\,p_k$ is finite. We shall denote it $(V,\E,\gma)$ where $\gma:\N \to V$, the spine, has $(\gma_n,\gma_{n+1})\in \E_n$ for every $n$. For tidiness, we'll denote $\gma_n$ by its position in $V_n$. As in Section \ref{sec:basex} let $(K_{n,i})_{n,i\in\N}$ be a doubly-indexed i.i.d.~sequence with $P(K_{n,i}=k)=p_k$ and let $(K_{n,i}^*)$ be i.i.d.~with $\pr(K_{n,i}^*=k)=k\,p_k \,  / \sum_j jp_j$, so that $K_{n,i}^*$ has the size-biased distribution of $K_{n,i}$, as in Definition \ref{def:sbrv}. Let $X_0=1$ and $\gma_0=1$, and let $X_1=K_{0,1}^*$, $\gma_1$ be uniform random on $\{1,\dots,X_1\}$ and $\E_0=\{((1,1),(2,i)\colon i\in\{1,\dots,X_1\}\}$. For $n\ge 1$, let
\begin{align*}
X_{n+1} & =(K_{n,\gma_n}^*-K_{n,\gma_n}) + \sum_{i =1}^{X_n}K_{n,i}  = K_{n,\gamma_n}^* + \sum_{i \ne \gamma_n} K_{n,i}\ \ \text{and} \\
c_{n,i} & =\1(i \ge \gma_n)(K_{n,i}^*-K_{n,i}) + \sum_{j=1}^i K_{n,j} = \1(i\ge \gamma_n) K_{n,i}^* + \sum_{j \le i,  \ j \ne \gamma_n} K_{n,j}.
\end{align*}
As before let $\E_n=\{(n,i),(n+1,j)\colon c_{n,i-1} < j \le c_{n,i}\}$ and let $\gma_{n+1}$ be uniform random on $\{c_{n,\gma_n-1},\dots, c_{n,\gma_n}\}$. Let $\E=\bigcup_n \E_n$. Then $(V,\E,\gma)$ is the SD.\\

\nid\tbf{Connection to the lookdown. }For a GW tree, the process of rescaled population sizes $(W_n)$ on $\R_+$ defined by $W_n=X_n/\mu^n$ is a martingale, so the measure of the GW tree biased on $(W_n)$, called the size-biased GW tree, can be defined. As explained in Section 2 of \cite{LPP}, the key property of the SD that facilitates its analysis is that for fixed $n$, its restriction to generations $0,\dots,n$ is equal in distribution to that of the graph obtained from a size-biased GW tree $(V,\E)$ by selecting a uniform random vertex in $V_n$ for $\gma_n$, and letting $(\gma_m)_{m\le n}$ be the unique path from $V_0$ to $\gma_n$.\\

Since $X_n$ is determined by $(k_m)_{m\le n}$, the effect of size-biasing is to multiply probabilities by some function of $(k_n)$. In particular, conditioned on $(k_n)$ the GW and size-biased GW trees have the same distribution, that we'll represent by $(V,\E)$, and have the same ordered lookdown, that we'll denote by $(V,\E')$. By Corollary \ref{cor:lkdn} there is a coupling $\E'=\alpha(\E)$ for which $\alpha_n$, and thus also $\alpha_n^{-1}$, is uniform and independent of $(\E_m)_{m<n}$. Let $\gma_n=\alpha_n^{-1}(1)$, so that $(\gma_n)$ is the preimage of the base path in the lookdown, then conditioned on $(\E_m)_{m<n}$, $\gma_n$ is uniform random on $V_n$. Giving $(k_n)$ the distribution assigned by the size-biased GW tree, $(V,\E,\gma)$ has the distribution of the SD since, from the previous paragraph, the distribution of its restriction to generations $0,1,\dots,n$ agrees with that of the SD for each $n$.\\

\nid\tbf{Fixation, takeover and identifiability. }In the above construction of the SD, all subtrees stemming from vertices not on the spine are (unbiased) GW trees, and independent for distinct vertices. Unbiased GW trees are a.s.~finite iff $\mu \le 1$, so for a size-biased GW tree, fixation has probability $1$ if $\mu\le 1$ and $0$ if $\mu>1$. In particular, from the above correspondence, together with Theorem \ref{thm:tkover}, Corollary \ref{cor:baseip} and the fact that fixation implies takeover, the spine, or equivalently the base path in the lookdown, is a.s.~identifiable if $\mu\le 1$. In \cite{sbsurv}, Theorem 7.4, it's shown that if $\mu>1$, the spine is identifiable with probability $0$ if $\sum_k k\log_+(k)p_k<\infty$ and with probability $1$ if $\sum_k k\log_+(k)p_k=\infty$, which is also the threshold that separates convergence of the rescaled size of the process to $0$ or a positive limit (see Kesten-Stigum theorem, described in Section \ref{sec:sd}). This is no coincidence, as both results hinge on the fact that the measure of the size-biased and unbiased trees are equivalent if $\sum_k k\log_+(k)p_k<\infty$ and mutually singular if $\sum_k k\log_+(k)p_k=\infty$; for the Kesten-Stigum theorem this can be used to deduce the behaviour of $W$ with respect to the unbiased measure using the ``key dichotomy'' ((3.2) in \cite{LPP}) while for identifiability, singularity means that it is possible to pick out a biased tree from among a collection of unbiased trees and thus identify the spine.\\

If we combine the identifiability result with Corollary \ref{cor:baseip}, the coalescent time scale $(t_n)$ has a finite limit a.s.~if $\sum_k k\log_+(k)p_k<\infty$ and diverges if $\sum_k k\log_+(k)p_k=\infty$, a fact that does not appear easy to verify directly, i.e., without using the connection between the coalescent time scale and identifiability. The result has the following very natural explanation: if $\mu>1$ but $\sum_k k\log_+(k)p_k<\infty$, population growth is fast and litter sizes are on average small enough that no single set of descendants is dominant, while if $\sum_k k\log_+(k)p_k=\infty$, size-biased litters are so large that they overwhelm all other growth, leading to takeover by the spine.


\section{Proofs of results from Section \ref{sec:prelim_neut}}\label{sec:neut}


As in \eqref{eq:xin}, let $\xi_n$ be a fixed partition with block sizes given by the non-zero entries of $k_n$. Let $\P_n$ be the set of permutations of $\{1,\dots,X_n\}$ and let $\sigma_n$ be uniformly distributed on $\P_n$. Write $\eqd$ to mean equality in distribution. Suppose $(k_n)$ is fixed. Say that $K_n$ is a uniform rearrangement of $k_n$, or \emph{$k_n$-uniform}, if $(K_n(i))_{i=1}^{X_n} \eqd (k_n(\sigma_n(i)))_{i=1}^{X_n}$, similarly, that $\Xi_n$ is \emph{$\xi_n$-uniform} if $\Xi_n \eqd \{\sigma_{n+1}(A)\colon A \in \xi_n\}$. Then, we have the following easily verified fact.

\begin{remark}\label{rem:unif}
Let $\xi_n$, $\sigma_n$ be as above. For a genealogy model based on $(k_n)$,
\enumrom
\item $K_n$ is exchangeable iff it is $k_n$-uniform.
\item $\Xi_n$ is exchangeable iff it is $\xi_n$-uniform.
\enumend
In particular, the model is
\enumrom
\item \emph{forward neutral} iff for all $n$, conditioned on $(\E_m)_{m<n}$, $(K_n(i))$ is $k_n$-uniform, and is
\item \emph{backward neutral} iff for all $n$, conditioned on $(\E_m)_{m>n}$, $\Xi_n$ is $\xi_n$-uniform.
\enumend
\end{remark}

For a fixed $k_n$, there is only one $k_n$-uniform distribution, i.e., knowing that $K_n$ is $k_n$-uniform fixes its distribution. We exploit this fact in the proof of Lemma \ref{lem:fwdcoup}, that describes how to rearrange one forward neutral model into another.

\begin{proof}[Proof of Lemma \ref{lem:fwdcoup}]
Since both models are forward neutral with the same $(k_n)$, Remark \ref{rem:unif} implies that both $K_0$ and $K_0'$ are $k_n$-uniform, and in particular $K_0 \eqd K_0'$, so let $\lph_0$ be the identity. Proceeding by induction, for $n\ge 0$ suppose $(\alpha_m)_{m\le n}$ are such that $((\alpha(\E_m))_{m<n},\lph_n(K_n))\eqd ((\E_m')_{m<n},K_n')$ and $(\alpha_m)_{m\le n}$ is determined by $(\E_m)_{m<n}$ and $(U_m)_{m < n}$. We'll complete the induction step as follows:
\enumar
\item Obtain $\alpha_{n+1}$ as a function of $U_n$, $(\alpha_m)_{m\le n}$ and $(\E_m)_{m\le n}$ such that $(\alpha(\E_m))_{m\le n} \eqd(\E_m')_{m\le n}$.
\item Show that, conditioned on $(\alpha(\E_m))_{m\le n}$, $\alpha_{n+1}(K_{n+1})$ is $k_n$-uniform.
\enumend
To see that this suffices, note that by Remark \ref{rem:unif}, conditioned on $(\E_m')_{m\le n}$, $K_{n+1}'$ is $k_n$-uniform, so the second half of 1 combined with 2 implies that
\[((\lph(\E_m))_{m\le n}, \lph_{n+1}(K_{n+1})) \eqd ((\E_m')_{m\le n},K_{n+1}'),\]
and the first half of 1 combined with the second half of the induction assumption implies that $\lph_{n+1}$ is a function of $(U_m)_{m\le n}$ and $(\E_m)_{m\le n}$, completing the induction step. It remains to show 1-2.\\

In order for $(\lph(\E_m))_{m\le n} \eqd (\E_m')_{m\le n}$ to follow from the induction assumption it suffices to use $\alpha_{n+1}$ to arrange the vertices in $V_{n+1}$ such that the distribution of $\lph(\E_n)$ conditioned on $((\alpha(\E_m))_{m<n}, \alpha(K_n))$ is correct, which can be done by taking $\alpha_{n+1}$ to be some function of $(\alpha(\E_m)_{m<n},\alpha(K_n))$ and an independent $\unif[0,1]$ random variable -- this is true for conditional distributions of any discrete random variables, or more generally those with regular conditional distributions. We claim that $U_n$ is independent of $(\alpha(\E_m)_{m<n},\alpha(K_n))$, and thus suitable for this role. To see this, note that $((\lph(\E_m))_{m\le n},\lph(K_n))$ is determined by $(\E_m)_{m\le n}$ and $(\lph_m)_{m\le n}$, and that by the induction assumption, the latter are determined by $(\E_m)_{m\le n}$ and $(U_m)_{m<n}$, of which $U_n$ is independent. It then follows also that $\lph_{n+1}$ is determined by $(\E_m)_{m\le n}$ and $(U_m)_{m\le n}$, completing 1.\\

For 2, first condition on the value of $(U_m)_{m\le n}$. This does not affect the joint distribution of $((\E_m)_{m\le n},K_{n+1})$, and $(\alpha_m)_{m\le n+1}$ is now determined by $(\E_m)_{m\le n}$. Since $\E$ is forward neutral and based on $(k_n)$, $K_{n+1}$ is $k_n$-uniform and independent of $(\E_m)_{m\le n}$. So, conditioned on both $(U_m)_{m\le n}$ and $(\E_m)_{m\le n}$, $K_{n+1}$ is $k_n$-uniform. Since $(\alpha_m)_{m\le n+1}$ is determined by $(U_m)_{m\le n}$ and $(\E_m)_{m\le n}$, and since the $k_n$-uniform distribution is invariant under independent permutation, conditioned on $(\alpha(\E_m))_{m\le n}$, $\alpha_{n+1}(K_{n+1})$ is $k_n$-uniform, which gives 2.
\end{proof}

Next we prove Theorem \ref{thm:neut}, that shows that scrambling a forward or backward neutral model gives the completely neutral model. Recall that $\E_n$ is exchangeable if $\E_n\eqd \{(\sigma(v),\sigma'(w))\colon (v,w) \in \E_n\}$ for every $\sigma \in \P_n$ and $\sigma'\in \P_{n+1}$. Suppose $(k_n)$ is fixed. Then, similarly to Remark \ref{rem:unif}, if $\E_n$ is exchangeable then it is $k_n$-uniform in the sense that $\E_n\eqd \{(\sigma(v),\sigma'(w)\colon (v,w) \in E\}$, where $E$ is any fixed set of edges on $V_n\times V_{n+1}$ that has unlabelled litter sizes $k_n$ (and each $w\in V_{n+1}$ has a unique $v\in V_n$ such that $(v,w)\in E$). In other words, when $k_n$ is fixed, knowing that $\E_n$ is exchangeable fixes its distribution.\\

Recall that $a(v)$ is the parent of $v$, and for each $n$ and $v\in V_n$ let $O(v)=\{w\in V_{n+1}\colon v=a(w)\}$ denote the offspring of $v$. It's easy to check that $\E_n$ is exchangeable iff
\enumrom 
\item $K_n$ is exchangeable and conditioned on $K_n$, $(O(v))_{v\in V_n}$ is exchangeable, or if
\item $\Xi_n$ is exchangeable and conditioned on $\Xi_n$, $(a(v))_{v\in V_{n+1}}$ is exchangeable.
\enumend

Recall that 
$\P$ is the set of permutations of $V$ that map each $V_n$ to itself, and that $\sigma\in \P$ is uniformly distributed if the restrictions $(\sigma_n)$ of $\sigma$ to $V_n$ are independent and uniformly distributed. Let $O_n = (O(v))_{v\in V_n}$ and $a_n = (a(v))_{v \in V_{n+1}}$. For $\sigma \in \P$, define $\sigma(K_n)$, $\sigma(\Xi_n)$ and $\sigma(\E_n)$ by $\sigma(K_n)(i) = K_n(\sigma_n(i))$, $\sigma(\Xi_n)=\{\sigma_{n+1}(A)\colon A \in \Xi_n\}$ and $\sigma(\E_n)=\{(\sigma_n(v),\sigma_{n+1}(w))\colon (v,w)\in\E_n\}$, similarly for $\sigma(O_n)$ and $\sigma(a_n)$.

\begin{proof}[Proof of Theorem \ref{thm:neut}]
We need to prove statements 1 and 2 of the theorem. For 1 it suffices to show either that
\enumrom
\item for each $n$, $\sigma(K_n)$ is exchangeable and independent of $(\sigma(\E_m))_{m<n}$, and \\conditioned on $(\sigma(\E_m))_{m<n}$ and $\sigma(K_n)$, $\sigma(O_n)$ is exchangeable, or
\item for each $n$, $\sigma(\Xi_n)$ is exchangeable and independent of $(\sigma(\E_m))_{m>n}$ and \\ conditioned on $(\sigma(\E_m))_{m>n}$ and $\sigma(\Xi_n)$, $\sigma(a_n)$ is exchangeable.
\enumend

We begin with the forward case and show (i). Roughly speaking, since $K_n$ is exchangeable and independent of $(\E_m)_{m<n}$, $\sigma_{n+1}$ can be used to randomize $O_n$. More precisely, since $K_n$ is exchangeable and independent of $(\E_m)_{m<n}$, and $\sigma$ is independent of $\E$, $\sigma(K_n)$ is exchangeable and independent of $(\sigma(\E_m))_{m<n}$. Moreover, $\sigma(K_n)$ is determined by $\E_n$ and $\sigma_n$. Since $\sigma_{n+1}$ is uniform and independent of $(\sigma_m)_{m\le n}$ and $\E$, conditioned on $(\sigma(\E_m))_{m<n}$ and $\sigma(K_n)$, $\sigma(O_n)$ is exchangeable, as desired. The backward case is analogous, showing (ii) instead of (i), so we omit it.\\

To obtain 2(a), first condition on $(\sigma_m)_{m\le n}$. Since $K_n$ is exchangeable, so is $\sigma(K_n)$. Since $\sigma_{n+1}$ is independent of $\E$, conditioned on $K_n$ and thus on $\sigma(K_n)$, $\sigma(O_n)$ is exchangeable. We have shown that conditioned on $(\sigma_m)_{m\le n}$, $\sigma(\E_n)$ is exchangeable. For $m>n$, continue as in the paragraph above to show that $\sigma(\E_m)$ is exchangeable and independent of $(\sigma(\E_j))_{j=n}^{m-1}$. Since exchangeability fixes the distribution of each $\sigma(\E_m)$, the desired result is proved, since we have shown the distribution of $(\sigma(\E_m))_{m\ge n}$ is the same, no matter the values of $(\sigma_m)_{m\le n}$. 2(b) is similar, beginning from $\Xi_n$ and working back, so we omit it.
\end{proof}

Finally we prove the forward-backward coupling. 

\begin{proof}[Proof of Theorem \ref{thm:fwdback}]
If $(V,\E)$ is completely neutral, the result follows almost immediately from Theorem \ref{thm:neut}: given backward neutral $\E'$, let $ \sigma \in \P$ be uniform and independent of $\E'$, then let $\E=\sigma(\E')$. By Theorem \ref{thm:neut}, $(\E_m)_{m<n}$ is independent of $(\sigma_m)_{m\ge n}$, so $\sigma_n$ is uniform and independent of $(\E_m)_{m<n}$. Let $\alpha=\sigma^{-1}$. Then $\E'=\alpha(\E)$, and for each $n$, $\alpha_n$ is uniform and independent of $(\E_m)_{m<n}$, as desired.\\

To achieve the coupling in general, we scramble $\E'$ to obtain $\E''$, then arrange $\E''$ into $\E$. Let $\sigma$ be uniform and independent of $\E'$ and let $\E''=\sigma(\E')$. Let $(U_n)$ be i.i.d.~$\unif[0,1]$, independent of both $\E'$ and $\sigma$ and let $\E=\beta(\E'')$ with $\beta$ as in Lemma \ref{lem:fwdcoup}.\\

Let $\F_n$ be the $\sigma$-algebra generated by $(\E_m'')_{m<n}$ and $(U_m)_{m<n}$, so that $(\beta_m)_{m\le n}$ is $\F_n$-measurable. Since $(\E_m)_{m<n}$ is determined by $(\E_m'')_{m<n}$ and $(\beta_m)_{m\le n}$, $(\E_m)_{m<n}$ is $\F_n$-measurable. Let $\alpha= (\beta \circ \sigma)^{-1}$, so that $\E'=\sigma(\E)$. By Theorem \ref{thm:neut}, $\sigma_n$ is uniform and independent of $(\E_m'')_{m<n}$. By definition of $(U_m)$, it follows that $\sigma_n$ is independent of $\F_n$. Since $\beta_n$ is $\F_n$-measurable and $\sigma_n$ is uniform and independent of $\F_n$, it follows that, conditioned on $\F_n$, $\alpha_n$ is uniform. In particular, $\alpha_n$ is uniform and independent of $\F_n$. Since $(\E_m)_{m<n}$ is $\F_n$-measurable the result is proved.
\end{proof}




\section{Size-biased ordering in the ordered lookdown}\label{sec:sbo}

Here we prove Proposition \ref{prop:lkdnsel}. We will occasionally abuse notation and identify $(n,i)\in V_n$ with $i\in \{1,\dots,X_n\}$, and view $\sigma_n$ as a permutation of $\{1,\dots,X_n\}$.

\begin{proof}[Proof of Proposition \ref{prop:lkdnsel}]
Let $\E=\alpha(\E')$ be the coupling of Corollary \ref{cor:lkdn} with $\E'$ completely neutral and $\E$ the lookdown, so that for each $m$, $\alpha_m$ is uniform and independent of $\F_m$, the $\sigma$-algebra generated by $(\E_j')_{j<m}$. Let $\Xi_{n,m},\Xi_{n,m}'$ correspond to $\E,\E'$ respectively, similarly for other objects, so that $\Xi_{n,m}'$ is $\F_m$-measurable and $\Xi_{n,m}=\{\alpha_m(A)\colon A \in \Xi_{n,m}'\}$. Since $\smash{(D_{n,m}(i))_{i=1}^{\ell_{n,m}}}$ sorts $\Xi_{n,m}$ by least element, $X_{n,m}(i)>0$ for $i\le \ell_{n,m}$, and conditioned on $\F_m$ it follows from Algorithm \ref{alg:sbo2} that $\smash{(\alpha_m^{-1}(D_{n,m}(i)))_{i=1}^{\ell_{n,m}}}$ is a size-biased ordering of $\Xi_{n,m}'$ and correspondingly that $\smash{(X_{n,m}(i))_{i=1}^{\ell_{n,m}}}$ is in size-biased order. By definition of $\ell_{n,m}$ and since $X_{n,m}(i)>0$ for $i\le \ell_{n,m}$, $X_{n,m}(i)=0$ for $i>\ell_{n,m}$, which implies the full vector $X_{n,m}$ is in size-biased order.\\

It remains to show the vector of limiting frequencies $(x((n,i)))_{i=1}^{X_n}$ in the ordered lookdown is in size-biased order, for which it suffices that $x(n,i)>0$ for $i\le \ell_n^*$ and $(x((n,i)))_{i=1}^{\ell_n^*}$ is in size-biased order. Now, $x_{n,m}'$ is $\F_m$-measurable, and by tracing ancestors, $\alpha_n$ is determined by $\F_m$ and $\alpha_m$. Let's estimate the conditional distribution of $\alpha_n$ given $\F_m$, using $\alpha_m$. From what we showed above, $(x_{n,m}(i))_{i=1}^{\ell_{n,m}}$ is in size-biased order, moreover, $x_{n,m}(i) = x_{n,m}'(\alpha_n(i))$. In particular, for any distinct $j_1,\dots,j_{\ell_n^*}$ for which $x_{n,m}'(j_i)>0$ for all $i\le \ell_n^*$,
\[\pr(\alpha_n(i)=j_i \ \forall i\le \ell_n^*\mid \F_m) = \prod_{i=1}^{\ell_n^*}\frac{x_{n,m}'(j_i)}{\sum_{a=i}^\ell x_{n,m}'(j_a)}.\]
Let $\F=\bigvee_m \F_m$. If, moreover, $x_{n,m}'(j_i)>0$ for all $m,i$, then from a.s.~convergence of frequencies (see Lemma \ref{lem:xmart}) and the bounded convergence theorem, letting $m\to\infty$ above,
\[\pr(\alpha_n(i)=j_i \ \forall i\le \ell_n^*\mid \F) = \prod_{i=1}^{\ell_n^*}\frac{x'(n,j_i)}{\sum_{a=i}^\ell x'(n,j_a)}.\]
The above is non-zero iff $x'(n,j_i)=0$ for all $i\le \ell_n^*$. Since there are $\ell_n^*$ values of $j$ for which $x'(n,j)>0$, it follows that $x((n,i))>0$ a.s.~for all $i\le \ell_n^*$, and from the formula that $(x((n,i)))_{i=1}^{\ell_n^*}$ is in size-biased order.
\end{proof}

\section{Takeover}\label{sec:dom}

In this section we prove Lemmas \ref{lem:coalprob}-\ref{lem:pnm_freq} and Theorem \ref{thm:tkover}. Recall that $s_n$ is the probability that a uniform random pair in $V_{n+1}$ has the same parent in $V_n$ and that $t_\infty=\sum_n s_n$ is the total time elapsed on the coalescent time scale.

\begin{proof}[Proof of Lemma \ref{lem:coalprob}]
We begin by deducing the last two statements, then prove the formula for $p_{n,m}$. For the second statement just note that $0\le s_j\le 1$ for all $j$. For the third statement, if $s_j=1$ for some $j\ge n$ then clearly $p_{n,m}=0$ for $m>j$, if $s_j \ge 1/2$ infinitely often then $\sum_j s_j=\infty$ and for any $n$, $\prod_{j=n}^\infty (1-s_j) \le 2^{-\infty}=0$, otherwise $0\le s_j<1/2$ eventually, and since for $0\le x<1/2$, $e^{-2x} \le 1-x\le e^{-x}$, $\prod_{j=n}^\infty (1-s_j)=0$ iff $\sum_j s_j=\infty$.\\

It remains to prove the formula for $p_{n,m}$. Since $p_{n,m}$ is invariant under permutation of edges, we may assume that $\E$ is completely neutral (see Definition \ref{def:unif}). For $j\le m$ let $a_j(v)$ denote the ancestor of $v$ in $V_j$, then for arbitrary $v,w \in V_m$, $p_{n,m}=1-\pr(a_n(v)\ne a_n(w))$. The event $a_j(v) \ne a_j(w)$ is determined by $\Xi_{j+1},\dots,\Xi_m$, of which $\Xi_j$ is independent. Since $\Xi_j$ is also uniform, it follows that
$$\pr(a_{j-1}(v) \ne a_{j-1}(w) \mid a_j(v) \ne a_j(w)) = 1-s_{j-1}.$$
Since $a_j(v)=a_j(w)$ implies $a_{j-1}(v)=a_{j-1}(v)$,
$$\pr(a_n(v) \ne a_n(w)) = \prod_{j=n}^{m-1} \pr(a_j(v) \ne a_j(w) \mid a_{j+1}(v) \ne a_{j+1}(w))$$ and the result follows.
\end{proof}

Our next task is to prove Lemma \ref{lem:pnm_freq}, for which the following concept will be useful. It is familiar in coalescent theory, see for example equation (30) in \cite{bercoal}.

\begin{definition}
Let $A$ be a finite set and $P$ a partition of $A$. Define the \emph{concentration} of $P$ by
\begin{align}\label{eq:conc}
c(P) = \pr(v\sim w)
\end{align}
where $(v,w)$ is a pair of elements of $A$ chosen uniformly at random without replacement and $v\sim w$ if $v,w$ are in the same element of $P$. 
\end{definition}
For example, $s_n=c(\Xi_n)$ and $p_{n,m}=c(\Xi_{n,m})$ for each $n,m$. Notice that $c(P)=0 \Lrarrow P = \{\{x\}\colon x \in A\}$ and $c(P)=1 \Lrarrow  P=\{A\}$ which justifies the name. We easily compute
\begin{align}\label{eq:concform}
c(P)=\dfrac{1}{\binom{|A|}{2}}\sum_{B\in P} \binom{|B|}{2} = \frac{\sum_{B\in P}|B|(|B|-1)}{|A|(|A|-1)}
\end{align}
which can be interpreted in the following useful way. If $B_I$ is a size-biased sample from $P$ then $P(I=i)=|B_i|/|A|$. Comparing to \eqref{eq:concform}, we see that
$$c(P)=\ev\left[ \, \frac{|B_I|-1}{|A|-1}\,\right],$$
and rearranging,
\begin{align}\label{eq:cPbias}
\ev[\,|B_I|\,] = 1 + (|A|-1)c(P).
\end{align}
Equation \eqref{eq:cPbias} is the main ingredient needed to complete the proof of Theorem \ref{thm:tkover}. However, we will also need to treat a special case, using the following lower bound on concentration.

\begin{lemma}\label{lem:clbnd}
Let $P$ be a partition of a finite set $A$ with $|A|\ge 2$. Then
$$c(P) \ge \frac{1/|P|-1/|A|}{1-1/|A|}.$$
\end{lemma}

\begin{proof}
Let $n=|P|$ and $N=|A|$. Order $P$ as $B_1,\dots,B_n$ and recall from just above \eqref{eq:cPbias} that
$$c(P)=\ev\left[ \, \frac{|B_I|-1}{|A|-1}\,\right],$$
where $B_I$ is a size-biased sample from $P$, i.e., $P(I=i)=|B_i|/N$. Let $J$ be uniform on $\{1,\dots,n\}$ so that $B_J$ is a uniform (unbiased) sample from $P$. Size-biased random variables stochastically dominate their unbiased counterparts (Section 2.2.4 in \cite{sbsurv}). Since
$\ev[\, B_J \,] = \sum_i |B_i|P(J=j) = \sum_i |B_i|/n = N/n$, $\ev[\, B_I\, ] \ge N/n$, so using the above display,
$$c(P) \ge \frac{N/n-1}{N-1} = \frac{1/n-1/N}{1-1/N}.$$
\end{proof}

The special case that requires  treatment is the following.

\begin{lemma}\label{lem:smallpop}
If $\max_i k_n(i)>1$ then $s_n\ge 1/X_n^2$. In particular, if $\tau=\infty$, $k_n \ne (1,\dots,1)$ for all $n$ and $\liminf_n X_n<\infty$ then $t_\infty=\infty$.
\end{lemma}

\begin{proof}
Recall that $\Xi_n$ partitions $V_{n+1}$ and that $\ell_n=|\Xi_n|$ is the number of vertices in $V_n$ with at least one child. In particular, $\ell_n \le X_n$ and if $\max_i k_n(i)>1$ then
$$X_{n+1}-\ell_n  = \sum_{i=1}^{X_n} (k_n(i) - \1(k_n(i)>0))\ge 1.$$
If $a$ is fixed then $b\mapsto (a-b)/(1-b)$ is decreasing on $(-\infty,\min(1,a)]$. Since $s_n=c(\Xi_n)$, using Lemma \ref{lem:clbnd} with $P=\Xi_n$ and $A=V_n$, $|P|=\ell_n$ and $|A|=X_{n+1}$ so
\[s_n = \frac{1/\ell_n-1/X_{n+1}}{1-1/X_{n+1}} \ge \frac{1/\ell_n - 1/(\ell_n+1)}{1-1/(\ell_n+1)} = \frac{1}{\ell_n^2} \ge \frac{1}{X_n^2}.\]
For the second statement, suppose $C:=\liminf_n X_n<\infty$. If $\max_i k_n(i)\le 1$, the assumption on $k_n$ implies $\min_i k_n(i)=0$ and thus $X_{n+1}<X_n$. If $\tau=\infty$ then $X_m\ge 1$ for all $m$, so if $X_n \le C$ it follows that $\max_i k_m(i)>1$ for some $n\le m <n+C$. Let $m$ be the least such value, then $X_m\le C$ and using the first statement, $s_m \ge 1/X_m^2 \ge 1/C^2$. Since $X_n\le C$ infinitely often, the result follows.
\end{proof}






We can now prove Lemma \ref{lem:pnm_freq}, which connects the limiting frequency on the base path to the coalescence probability $p_{n,\infty}$ from Lemma \ref{lem:coalprob}.

\begin{proof}[Proof of Lemma \ref{lem:pnm_freq}]
Combining Proposition \ref{prop:lkdnsel} with \eqref{eq:cPbias}, for each $n<m$,
$$\ev[X_m(u_n)] = 1 + (X_m-1)p_{n,m}$$
and dividing by $X_m$,
\begin{align}\label{eq:esbfreq}
\ev[x_m(u_n)] = 1/X_m + (1-1/X_m)p_{n,m}.
\end{align}
By Lemma \ref{lem:xmart}, $\ev[x_m(u_n)]\to \ev[x(u_n)]$ as $m\to\infty$ and by Lemma \ref{lem:coalprob}, $p_{n,\infty}= \lim_{m\to\infty} p_{m,n}$ exists, and is equal to $1$ if $t_\infty=\infty$. Since $\tau=\infty$ by assumption, $X_m\ge 1$ for all $m$, so $1/X_m$ is bounded. If $1/X_m \to 0$ then taking $m\to\infty$ in \eqref{eq:esbfreq} gives \eqref{eq:esbasymfreq} in that case. If instead $\limsup_m 1/X_m>0$, then Lemma \ref{lem:smallpop} implies $t_\infty=\infty$, so $p_{n,\infty}=1$, and the RHS of \eqref{eq:esbfreq} is $(1-p_{n,m})/X_m+p_{n,m}$. As $m\to\infty$, $p_{n,m} \to 1$ and $1/X_m$ is bounded, so the RHS of \eqref{eq:esbfreq} tends to $1=p_{n,\infty}$.
\end{proof}

\begin{proof}[Proof of Theorem \ref{thm:tkover}]

If $t_\infty=\infty$ then $p_{n,\infty}=1$ and by \eqref{eq:esbasymfreq}, $\ev[x(u_n)]=1$, and since $x(u_n)\in[0,1]$, $x(u_n)=1$ a.s. If $t_\infty<\infty$ then $\sum_{m\ge n}s_m\to 0$ as $n\to\infty$, and \eqref{eq:coalprob} implies $p_{n,\infty} \to 0$ as $n\to\infty$. Notice that if $x_I$ is a size-biased sample from $(x_i)_{i=1}^\ell$ as in Definition \ref{def:sb} then
$$\ev[x_I] = \sum_i x_iP(I=i) = \frac{\sum_i x_i^2}{\sum_j x_j}.$$
From Proposition \ref{prop:lkdnsel}, $x(u_n)$ is a size-biased sample from $(x(v))_{v \in X_n}$, and $\sum_{v \in V_n}x(v)=1$, so
$$p_{n,\infty}=\ev[x(u_n)] = \ev[\sum_{v \in V_n}x(v)^2] \ge \ev[\max_{v \in V_n}x(v)^2],$$
and $p_{n,\infty}\to 0$ implies $\max_{v \in V_n}x(v) \to 0$ in probability as $n\to\infty$.
\end{proof}

\section{Asynchronous models}

Here we prove Theorems \ref{thm:asynfs} and \ref{thm:tovar} (in reverse order, since \ref{thm:asynfs} makes use of \ref{thm:tovar}).

\begin{proof}[Proof of Theorem \ref{thm:tovar}]
We first prove a general form of (i). If $X_n>0$ and $X_{n+1}\ge \alpha X_n$ with $\alpha>1$ then $\Delta_n \ge \log \alpha>0$ and $ b_n-1=X_{n+1}-X_n\ge X_{n+1}(1-1/\alpha)$ so
\[s_n \ge \frac{(b_n-1)^2}{X_{n+1}^2} \ge (1-1/\alpha)^2>0.\]
So, if $X_{n+1}\ge \alpha X_n$ infinitely often then $\sum_n \Delta_n^2=\infty$ and $t_\infty=\infty$. \\

Next, we show (ii): suppose $X_n>1$ and $X_{n+1} \le 2X_n$, and let $a_n=\max(0,(X_{n+1}/X_n-1)^2)$. By concavity $x\log 2 \le \log(1+x) \le x$ for $x \in [0,1]$, which implies $a_n (\log 2)^2 \le \Delta_n \le a_n$. If $b_n\in\{0,1\}$ then $s_n=0$ and $X_{n+1}\le X_n$ so $\Delta_n=0$. Otherwise, $b_n\ge 2$ and $X_{n+1}>X_n\ge 2$ so $b_n-1\ge b_n/2$ and $X_{n+1}-1\ge X_{n+1}/2$, implying the right-hand inequality in
\[\frac{(b_n-1)^2}{X_{n+1}^2} \le s_n \le \frac{2(b_n-1)^2}{X_{n+1}^2/2}\]
(the left-hand one is obvious). Moreover, $b_n-1=X_{n+1}-X_n$ so $X_{n+1}\ge X_n$, and by assumption $X_{n+1}\le 2X_n$. Substituting for $b_n-1$ and using both inequalities yields
\[\frac{(X_{n+1}-X_n)^2}{4X_n^2} \le s_n \le \frac{4(X_{n+1}-X_n)^2}{X_n^2}.\]
The left- and right-hand sides are $a_n/4$ and $4a_n$, respectively. Combining with the inequalities between $a_n$ and $\Delta_n$ gives (ii). It remains to show that if $X_n>1$ eventually then $t_\infty=\infty$ iff $\sum_n \Delta_n^2=\infty$. If $X_{n+1}\ge 2X_n$ infinitely often then $s_n\ge 1/4$ and $\Delta_n\ge \log 2$ infinitely often and so $t_\infty=\sum_n \Delta_n^2=\infty$. Otherwise, let $n_0$ be such that $X_n\ge 2$ and $X_{n+1} \le 2X_n$ for all $n\ge n_0$. Then $\sum_{n<n_0}s_n$ and $\sum_{n<n_0}\Delta_n^2$ are finite, and from (ii), $\sum_{n\ge n_0}s_n=\infty$ iff $\sum_{n\ge n_0}\Delta_n^2=\infty$.
\end{proof}

\begin{proof}[Proof of Theorem \ref{thm:asynfs}]
\nid\tbf{Case (i).} Let $A:=\{n\colon b_n\ne 1\}$ and let $(m_n)$ enumerate $A$ in increasing order. Then by assumption $\liminf_n m_n/n>0$; denoting it $\alpha$, if $C=\limsup_n X_n/\sqrt{n}$ then $C/\alpha \ge \limsup_n X_{m_n}/\sqrt{m_n}$, so without loss of generality we may assume that $b_n\ne 1$ for all $n$. Furthermore we may assume that $X_n\to \infty$, since if $X_n\le C$ infinitely often then since $X_n>0$ and $b_n \ne 1$ for all $n$, it's easy to show that $s_n \ge 1/(C(C-1))$ infinitely often. Finally we may assume that $X_{n+1}\le 2X_n$ eventually, since otherwise $t_\infty=\infty$ by Theorem \ref{thm:tovar}. Let $n_0$ be such that $X_{n+1}\le 2X_n$ for $n\ge n_0$; we will show that
\[t_\infty \ge \sum_{n\ge n_0} \frac{1}{4X_n^2},\]
then the result follows from $X_n=O(\sqrt{n})$ and divergence of $\sum_{n\ge 1} 1/n$. Let $A_0=\{n\colon b_n=0\}$ and $A_2=\{n\colon b_n\ge 2\}=\N\setminus A_0$. If $m\in A_2$ and $m\ge n_0$, since $b_m\ge 2$ and $b_m-1=X_{m+1}-X_m$,
\[s_m \ge \frac{2(X_{m+1}-X_m)}{X_{m+1}^2} \ge \frac{X_{m+1}-X_m}{2X_m^2}.\]
For $n\in A_0$ let $M(n)=\min\{m>n\colon X_{m+1}\ge X_n\}$ and for $m\in A_2$ let $N(m)=\{n\colon M(n)=m\}$. By definition $X_m\le X_n$ for all $n\in N(m)$. Moreover, $|N(m)| \le X_{m+1}-X_m$ [add picture]. So, if $m\in A_2$ and $m\ge n_0$ then
\[s_m \ge \frac{1}{4X_m^2} + \frac{|N(m)|}{4X_m^2} \ge \frac{1}{4X_m^2} + \sum_{n\in N(m)} \frac{1}{4X_n^2}.\]
Since $M(n)$ exists for each $n\in A_0$, $\N=A_2 \cup \bigcup_{m \in A_2}N(m)$. In particular,
\[t_\infty \ge \sum_{m\in A_2, \ m\ge n_0}s_m \ge \sum_{n\ge n_0} \frac{1}{4X_n^2}.\]
\nid\tbf{Case (ii).} As above let $Y_n=\log X_n$ so that $\Delta_n=(Y_{n+1}-Y_n)_+$. If $y_0<y_n$ are fixed and $y_1,\dots,y_{n-1}$ are variable then
\[s(y) := \sum_{i=0}^{n-1} \max(0,y_{i+1}-y_i)^2\]
has the unique minimizer $y_i=y_0+(y_n-y_0)i/n$ and minimum value $(y_n-y_0)^2/n$. We'll prove this in a moment; let's first verify the desired result. If $Y_n-Y_m\ge (\ep/2)\sqrt{n}$ then $\sum_{i=m}^{n-1} \Delta_i^2 \ge (Y_n-Y_m)^2/n \ge \ep^2/4$. By assumption, for some $\ep>0$ $Y_n\ge \ep\sqrt{n}$ infinitely often. Let $n_0=0$ and recursively let $n_k=\min\{n>n_{k-1}\colon Y_n - Y_{n_{k-1}} \ge (\ep/2)\sqrt{n}\}$. Then
\[\sum_n \Delta_n^2 = \sum_{k=1}^\infty \sum_{n=n_{k-1}}^{n_k-1}\Delta_n^2 = \sum_{k=1}^\infty \ep^2/4=\infty.\]
It remains to examine $s(y)$. First, any minimizer has $y_i\le y_{i+1}$ for all $i$, shown in two steps:
\enumrom
\item if $y_i>y_n$ for some $i$ then, taking the least such $i$, $y_{i-1} \le y_n<y_i$ and $s(y)$ does not increase if $y_i,\dots,y_{n-1}$ are replaced with $y_n$, then
\item if $y_i<y_{i-1}$ for some $i$ then from step (i), $i<n$. Take the least such $i$ and let $j=\min\{k\ge i\colon y_k\ge y_{i-1}\}$; from step (i), $j$ exists and $j\le n$. Replacing $y_i,\dots,y_{j-1}$ with $y_{i-1}$ does not increase $s(y)$, and does increase the least $i$ such that $y_i<y_{i-1}$, if such an $i$ still exists.
\enumend
Repeating step (ii) at most finitely many times we obtain $y_1,\dots,y_{n-1}$ such that $y_0\le y_1\le \dots \le y_{n-1}\le y_n$, without increasing the value of $s(y)$. So, without loss of generality we may assume that $s(y)=\sum_{i=0}^{n-1} (y_{i+1}-y_i)^2$. Its gradient has entries $(\nabla s(y))_i = -2(y_{i+1}-y_i)+2(y_i-y_{i-1})$ which are zero when $y_{i+1}-y_i=y_i-y_{i-1}$ for $i=1,\dots,n-1$, yielding $y_i=y_0+(y_n-y_0)i/n$. Its Laplacian is $4(n-1)$ which is positive, ensuring the unique critical point is a minimizer.
 \end{proof}

\section{Fixation}\label{sec:fix}

We now give conditions for fixation, i.e., uniqueness of the infinite path. Recall the extinction time $\tau(v)=\inf\{m\colon D_m(v)=\emptyset\}$ of $v$, from Corollary \eqref{cor:tauord}. If $\tau=\infty$ ($X_n>0$ for all $n$) the ordered lookdown has at least one infinite path, the base path $((n,1)\colon n\ge 0)$, so fixation occurs if $\tau((n,i))<\infty$ for all $n$ and $i\ge 2$.  To show this, the basic principle that we'll use is that on the (truncated) coalescent time scale the min path (least-ranked descendant) of particles dominates a Markov chain $Y$ with $Y\to Y+1$ at rate $c (Y-1)^2$ for some $c>0$, which blows up in finite time if $Y(0)\ge 2$. To see where this comes from, consider the ordered lookdown for the Moran model. In this model, at rate 1 for each pair of positive integers $i<j$, the particle at $i$ has an offspring at $j$, and for each $k>j$ the particle at $k$ is ``pushed up'' to $k+1$, recall Figure \ref{fig:mrngrph}. Tracing ancestors backwards in time, each pair of particles gains a common ancestor at rate 1 (since this is the rate of coalescence, from any two locations), so the coalescent time scale is just the (actual) time scale $t$ of the process. If $D_t((s,i))$ denotes descendants, at time $t\ge s$, of the particle located at $i$ at time $s$, then $Y(t):=\min D_t((s,i))$ follows the path of that particle as it gets pushed up over time, if it does. A particle at $k$ gets pushed up at a total rate $\#\{(i,j)\colon 1 \le i < j\le k\}=k(k-1)/2$, so $Y\to Y+1$ at rate $Y(Y-1)/2$. If $Y(s)=i\ge 2$ then the expected time for $Y\to\infty$, which is also $\ev[\tau((s,i))-s]$, is thus
\[\sum_{j=i}^\infty \frac{2}{j(j-1)} = \frac{2}{i-1}<\infty.\]
In particular, the clade of any $(s,i)$ with $i\ge 2$ goes extinct in finite time almost surely, i.e., fixation occurs. Returning to the present context, we begin by considering asynchronous models (see Definition \ref{def:asyn}).

\begin{proof}[Proof of Theorem \ref{thm:fix}, asynchronous case]

Recall the min path $\gma_m(v)$ is defined for $m\ge n$ by $\gma_m(v) = \min D_m(v)$, with the convention $\min\emptyset=\infty$, and $\ell_n=|\Xi_n|$ is the number of individuals in generation $n$ that have offspring. Let $M(v) = \inf\{m\ge n\colon \gma_m(v) > \ell_m\}$. Then $\tau(v) = 1+M(v)$, since (i) if $m<M(v)$ then $\gma_m(v) \le \ell_m$ so some individual in $D_m(v)$ has offspring, while (ii) no individuals in $D_{M(v)}(v)$ have offspring. By assumption, $t_\infty^o=\infty$ so if $t_{M(v)}^o<\infty$ then $M(v)<\infty$ and thus $\tau(v)<\infty$, so it suffices to show $t_{M((n,i))}^o<\infty$ for all $n$ and $i\ge 2$.\\

Given $(n,i) \in V$ with $i\ge 2$, define $(Y(m))_{m\ge n}$ by $Y(m)=\gma_m((n,i))$ and let $M=M((n,i)) = \inf\{m\ge n\colon Y(m)>\ell_m\}$. Let $m(j)=M\wedge\inf\{m\ge n\colon Y(m)\ge j\}$ and $W_j = t_{m(j+1)}^o-t_{m(j)}^o$. Then, $t_M^o= \sum_{j=2}^{\infty} W_j$ so it is enough to show that $\sum_{j=2}^{\infty} \ev[\,W_j\,]<\infty$. \\

In the asynchronous case, $k_n = (b_n,1,\dots,1)$ so \eqref{eq:sntrunc} gives
$$s_m^o=\frac{2\, \1(b_m>0)}{X_{m+1}(X_{m+1}-1)}.$$
For $m$ such that $b_m\ge 2$, let $O_m$ denote the unique block of $\Xi_m$ with $b_m$ elements, which for the lookdown is a uniform $b_m$-subset of $\{1,\dots,X_{m+1}\}$. If $Y(m)>\ell_m$ then $Y(m+1)=\infty$, while if $Y(m)=j \le \ell_m$ then $Y(m+1)>Y(m)$ iff $|O_m \cap \{1,\dots,Y(m)\}| \ge 2$, so if $q_m(j) = \pr(Y(m+1)>j \mid  Y(m)=j)$ then from the distribution of $O_m$
\begin{align}\label{eq:qlb}
q_m(j) \ge \frac{j(j-1)}{X_{m+1}(X_{m+1}-1)} = j(j-1)\, s_m^o/2
\end{align}
and if $b_m<2$ then $s_m^o=0$, in which case \eqref{eq:qlb} also holds. In the lookdown, $(\E_m)$ are independent by construction, so for integer $a>0$
$$\pr(Y(a)=j \mid Y(m')=j, \ m\le m'< a) = 1-q_{a-1}(j),$$
and therefore
$$
\pr(m(j+1) > a \mid m(j)=m)  
= \prod_{m'=m}^{a-1} (1-q_{m'}(j)).
$$

Using $1-x\le e^{-x}$, then \eqref{eq:qlb}, for any $i\ge 1$
\begin{align}\label{eq:asynfix}
\prod_{m'=m}^{a-1} (1-q_{m'}(j))  &\le \exp\left(-\sum_{m'=m}^{a-1} q_{m'}(j)\right) \nonumber  \\
&\le \exp \left( -j(j-1) (t_a^o-t_m^o)/2\right).\end{align}

For each $m$, $X_{m+1} \ge \ell_m$ (each litter at time $m$ contributes at least $1$ to $X_{m+1}$), so if $Y(m) > X_{m+1}$ then $Y(m) > \ell_m$ and $M\le m$. In other words, if $m<M$ then $Y(m) \le X_{m+1}$. In particular, if $m(j) \le m < M$ then $s_m^o \le 2/(j(j-1))$. 
By definition,
\begin{align}\label{eq:asynfix1}
\pr(W_j>t_M^o-t_m^o \mid m(j)=m)=0
\end{align}
and if $m(j)=m$ and $W(j)>s$ for some $s<t_M^o-t_m^o$ then letting $a = \inf\{m'\colon t_{m'}^o-t_m^o > s\}-1$, $a<\min(m(j+1),M)$ so by the above bound on $s_m^o$,
\[t_a^o-t_m^o = t_{a+1}^o-t_m^o-s_a^o \ge s - 2/(j(j-1)).\]
Using \eqref{eq:asynfix}, we then have
\begin{align}\label{eq:asynfix2}
\pr(W_j>s \mid m(j)=m) &\le \pr(m(j+1)>a \mid m(j)=m) \nonumber \\
&\le \exp(-j(j-1)(t_a^o-t_m^o)/2) \nonumber \\
& \le  \exp(-j(j-1)s/2 + 1). 
\end{align}
Combining \eqref{eq:asynfix1} and \eqref{eq:asynfix2}, $\pr(W_j>s \mid m(j)=m) \le \exp(-j(j-1)s+1)$ for all $s>0$, so the same is true without conditioning on the value of $m(j)$. Then,
\begin{align*}
\ev[\,W_j\,] &= \int_0^{\infty} \pr(W_j>s)ds \\
& \le \int_0^{\infty} \exp(-j(j-1)s/2 +1)ds \\
& \le \frac{2e}{j(j-1)}.
\end{align*}
In particular, $\sum_{j\ge 2} \ev[\,W_j\,]<\infty$, completing the proof.
\end{proof}

To treat the general case we'll need a few auxiliary results. For the following, which uses size-biased ordering, it's convenient to think of a list $(x_1,\dots,x_\ell)$ as a multiset $S$, which is a set in which the same element can appear more than once. Then, $(y_i)_{i=1}^\ell$ is a size-biased ordering of $S$ if for each $k$ for which $S\setminus \{y_i\}_{i< k}$ has a non-zero element,  $y_k$ is a size-biased sample from $S\setminus \{y_i\}_{i<k}$.

\begin{lemma}\label{lem:sbosd}
Let $(y_i)_{i=1}^\ell$ be a size-biased ordering of a multiset $S\subset \R_+$ and let $(u_i)_{i=1}^\ell$ be a uniform random ordering of $S$. For any $c>0$, there is a coupling of $(y_i)$ and $(u_i)$ such that for all $j$, $N(j):=\#\{i\le j\colon y_i\ge c\} \ge M(j):=\#\{i\le j\colon u_i\ge c\}$.
\end{lemma}

\begin{proof}
Let $n=\#\{x \in S\colon x\ge c\}$. Given $j$, suppose a coupling of $(y_i)_{i=1}^j$ and $(u_i)_{i=1}^j$ exists with $N(i)\ge M(i)$ for $i\le j$, which holds trivially for $j=0$ with $N(0)=M(0)=0$. Let $\F_j$ be the $\sigma$-algebra generated by $(y_i)_{i=1}^j,(u_i)_{i=1}^j$. To specify the distribution of $y_{j+1},u_{j+1}$ conditioned on $\F_j$, there are two cases.\\

\nid\tbf{Case 1: $N(j)>M(j)$. } In this case $N(j+1) \ge M(j+1)$, so just let $y_{j+1},u_{j+1}$ be conditionally independent given $\F_j$, with the correct marginal distributions.\\ 

\nid\tbf{Case 2: $N(j)=M(j)$. }Let $z_j=\sum_{x \in S \setminus \{y_i\}_{i\le j}} x$ and $z_j(c)=\sum_{x \in S\setminus \{y_i\}_{i\le j}} x\1(x\ge c)$, then
$$\pr(y_{j+1} \ge c \mid \F_j) = \frac{z_j(c)}{z_j}.$$
Notice that $z_j(c)\ge c(n-N(j))$ and $z_j-z_j(c) \le c(\ell-j-(n-N(j))$. In particular,
\begin{align}\label{eq:sbosd1}
\pr(y_{j+1}\ge c\mid \F_j) \ge \frac{n-N(j)}{\ell-j}\end{align}
On the other hand,
\begin{align}\label{eq:sbosd2}
\pr(u_{j+1} \ge c \mid \F_j) = \frac{n-M(j)}{\ell-j}.
\end{align}
To couple $y_{j+1}$ and $u_{j+1}$, let $U$ be uniform on $[0,1]$ independent of $\F_j$. Arrange $S\setminus \{y_i\}_{i\le j}$ and $S\setminus \{u_i\}_{i\le j}$ in decreasing order as $(a_i)_{i=1}^{\ell-j}$ and $(b_i)_{i=1}^{\ell-j}$ and let $\alpha_i=\sum_{k=1}^i a_k/z_j$ and $\beta_i=i/(\ell-j)$, with $\alpha_0=\beta_0=0$. Then, let $y_{j+1}=a_i$ if $\alpha_{i-1} \le U < a_i$ and $u_{j+1}=b_i$ if $\beta_{i-1} \le U < \beta_i$. Let $I=\max\{i\colon a_i\ge c\}$ and $J=\max\{i\colon b_i \ge c\}$. If $N(j)=M(j)$ then using \eqref{eq:sbosd1} and \eqref{eq:sbosd2},
$$\alpha_I =\pr(y_{j+1} \ge c \mid \F_j) \ge \pr(u_{j+1} \ge c \mid F_j)=\beta_J$$ which means that if $u_{j+1}\ge c$ then $y_{j+1}\ge c$, and thus $N(j+1)\ge M(j)$.
\end{proof}

The next result will help us to show that the min path $Y$ of particles (as defined in the proof of the asynchronous case) tends to $\infty$ in finite time, on the truncated coalescent time scale, in the case where $Y$ has large jumps.

\begin{lemma}\label{lem:divseq}
Let $(a_i), (s_i)_{i\in \N}$ be sequences in $\R_+\cup\{\infty\}$ and let $y_0>0$, and for $i\in \N\cup\{\infty\}$ let $y_i=y_0 + \sum_{j\le i}a_j$ and $t_i=\sum_{j \le i}s_j$, so that $t_0=0$. Let $I=\inf\{i\colon y_i=\infty\}$, with $\inf\emptyset=\infty$. Suppose there is $c>0$ such that $a_{i+1} \ge  c \,y_i^2 s_i/2$ and $s_i \le 1/y_i$ for each $i$. Then $t_I \le 2+4/c$.
\end{lemma}

\begin{proof}
For each $i$, $y_{i+1}/y_i \ge 1+c\, y_is_i/2$. Let $I_0=0$ and $I_k=\inf\{i\colon y_i\ge 2^k\}$, so that $I=\lim_{k\to\infty}I_k$. If $I_{k+1}=I_k$ then $t_{I_{k+1}}-t_{I_k}=0$. Otherwise, for $I_k \le i < I_{k+1}$, (i) $y_{i+1}/y_i \ge 1+2^{k-1} c \, s_i$ and (ii) $s_i \le 2^{-k}$. Using (i),
$$2 \ge \frac{y_{I_{k+1}-1}}{y_{I_k}} \ge 1 + c\,2^{k-1}(s_{I_k}+\dots+s_{I_{k+1}-2}) = 1 + c\,2^{k-1}(t_{I_{k+1}-1}-t_{I_k}),$$
so $t_{I_{k+1}-1}-t_{I_k}\le 2^{1-k}/c$. Combining with (ii),
$$t_{I_{k+1}}-t_{I_k} = s_{I_{k+1}-1} + (t_{I_{k+1}-1}-t_{I_k}) \le 2^{-k} + 2^{1-k}/c =(1+2/c)\cdot 2^{-k}.$$
Therefore $t_I = \sum_k t_{I_{k+1}}-t_{I_k} \le (1+2/c) \sum_k 2^{-k} = 2+4/c$ as desired.
\end{proof}

The next result will allow us to extract a diverging subsequence of increments on the coalescent time scale, in the case where $Y$ can have large jumps, along which $Y$ has large enough jumps that Lemma \ref{lem:divseq} can be applied.

\begin{lemma}\label{lem:binsum}
Fix $p>0$ and let $(u_i)$ be i.i.d.~with $\pr(u_i=1)=1-\pr(u_i=0)=p$, and suppose $(s_i)$ is a sequence in $\R_+$ with $\sum_i s_i=\infty$. Then $\pr(\sum_i s_iu_i=\infty)=1$.
\end{lemma}

\begin{proof}
Let $I_0=\{i\colon s_i \ge 1\}$ and for $k\ge 1$, $I_k=\{i\colon 3^{-(k-1)}> s_i \ge 3^{-k}\}$, and let $\ep_k = \sum_{i\in I_k} s_i$ so that $\infty=\sum_i s_i = \sum_k \ep_k$. If $S\subset \N$ is infinite then $\pr(\sum_{i\in S}u_i=\infty)=1$ by the second Borel-Cantelli lemma, so if $\ep_0=\infty$ then $|I_0|=\infty$ and $\sum_i s_iu_i \ge \sum_{i\in I_0} u_i =\infty$. Otherwise, $\sum_{k\ge 1}\ep_k=\infty$. Let $K=\{k\ge 1\colon |I_k| > k\}$. If $k\notin K$ then $\ep_k \le k3^{-(k-1)}$ so $\sum_{k\notin K}\ep_k<\infty$ which means $\sum_{k\in K}\ep_k=\infty$. Let $N_k=\#\{i \in I_k\colon u_i=1\}$, then $N_k \eqd \bin(|I_k|,p)$ and $(N_k)_{k\ge 1}$ are independent. Let $A_k$ be the event that $N_k < |I_k|p/2$. A standard large deviations estimate gives $\lambda\in(0,1)$ such that
$$\pr(A_k) \le \lambda^{|I_k|}.$$
If $k\in K$ then $\pr(A_k) \le \lambda^k$, so $\sum_{k\in K}\pr(A_k)\le 1/(1-\lambda)<\infty$ and the first Borel-Cantelli lemma implies there is an a.s.~finite $k_0$ such that $N_k \ge |I_k|p/2$, and thus $\sum_{i \in I_k}s_i u_i \ge 3^{-k} |I_k|p/2$, for all $k\ge k_0$. Since $\ep_k \le 3^{-(k-1)}|I_k|$, it follows that
$$\sum_i s_iu_i \ge \sum_{k \in K\colon k\ge k_0} 3^{-k}|I_k|p/2 \ge \frac{p}{6}\sum_{k \in K\colon k \ge k_0} \ep_k =\infty.$$
\end{proof}

\begin{lemma}\label{lem:pebnd}
Let $Z\ge 0$ be such that $\var(Z) \le \ev[Z]$. Then
$$\pr(Z>\ev[Z]/2) \ge (\ev[Z] \wedge 18)/36.$$

\end{lemma}

\begin{proof}
Let $\mu=\ev[Z]$, then $\ev[Z^2]=\var(Z)+(\ev[Z])^2$ which, by assumption, is $\le \mu(1+\mu)$. Using Chebyshev's inequality,
$$\pr(Z \le \ev[Z]/2) \le \pr (|Z-\ev[Z]| > \ev[Z]/2) \le \frac{4\var(Z)}{(\ev[Z])^2} \le 4/\mu$$
which is at most $1/2$ if $\mu \ge 8$. Using the Paley-Zygmund inequality,
$$\pr(Z>\theta\,\ev[Z]) \ge (1-\theta)^2 \frac{(\ev[Z])^2}{\ev[Z^2]} \ge (1-\theta)^2\, \frac{\mu^2}{\mu(1+\mu)} = \frac{(1-\theta)^2\, \mu}{1+\mu}$$
Let $\theta=1/2$, then $\pr(Z>\ev[Z]/2) \ge \mu/(4(1+\mu))$ which is at least $\mu/36$ if $\mu\le 8$.

\end{proof}

\begin{proof}[Proof of Theorem \ref{thm:fix}, general case]
Given $(n,i)\in V$ with $i\ge 2$, define $(Y(m))_{m\ge n}$ and $M$ as in the asynchronous case. Begin by separating the coalescent time scale into two parts, one for small increments and one for large, as follows:
\begin{align*}
&s_m^s := s_m^o \1( s_m^o \le 2/(Y(m)(Y(m)-1))),\\
&s_m^\ell := s_m - s_m^s,
\end{align*}
then let $t_m^s = \sum_{i<m}s_m^s$ and $t_m^\ell = \sum_{i<m} s_m^\ell$, so that $t_m^o=t_m^s+t_m^\ell$. By assumption, $t_\infty=\infty$ so either $t_\infty^s=\infty$ or $t_\infty^\ell=\infty$.\\

Next, we use Algorithm \ref{alg:sbo1} to estimate $Y(m+1)-Y(m)$. Recall that $\Xi_m\eqd\{\sigma(A)\colon A\in \xi_m\}$ where $\xi_m$ is a fixed partition of $\{1,\dots,X_{m+1}\}$ with block sizes $k_m$ and $\sigma \in \P_{m+1}$ is uniform. Let $\xi_m(i)$ denote the block containing $i$, let $\alpha=\sigma^{-1}$ and let $I(1)=1$, $I(j)=\inf\{i>I(j-1)\colon \alpha(i)\notin \bigcup_{j'<j}\xi_m(I(j'))\}$. If $B_1,\dots,B_{\ell_m}$ orders $\Xi_m$ by least element, then $B_j=\sigma(\xi_m(I(j)))$ for each $j$. By definition of the min path, $Y(m+1)=\min B_{Y(m)}$, or equivalently, $Y(m+1)=I(Y_m)$.\\

For $N,K,n \in \N$ with $K,n \le N$, the hypergeometric distribution $\hyp(N,K,n)$ describes the number of successes when drawing $n$ objects from a population of size $N$ without replacement, in which there are $K$ objects that count as a successful draw. If $Z\eqd \hyp(N,K,n)$ then $E[Z]=nK/N$ (ignoring other draws, each draw has probability $K/N$ of being a success), and $\var[Z] \le \ev[Z]$ since the outcomes of distinct draws are negatively correlated (or check the well-known formulae). In particular, Lemma \ref{lem:pebnd} applies to hypergeometric random variables. \\

Let $(\F(i))$ be the natural filtration of $(\alpha(i))$, fix $j$, and suppose $Y(m)=j$. Since $\xi_m$ has $\ell_m$ blocks, and has $L_m$ blocks of size at least $2$, by Lemma \ref{lem:sbosd}, $N:=\#\{j' \le \lf j/2\rf\colon |\xi_m(I(j'))| \ge 2\}$ dominates $\hyp(\ell_m,L_m,\lf j/2\rf)$. Let $S=\bigcup_{i\le\lf j/2\rf} (\xi_m(I(i)) \setminus \{I(i)\})$, then $|S| \ge N$. Let $N' = \#\{j'\le I(j)\colon \alpha(j') \in S\}$, then $I(j) \ge j + N'$, so if $Y(m)=j$ then $Y(m+1)=I(j) \ge Y(m)+N'$. Conditioned on $\F(I(\lf j/2\rf))$, which determines $N$, $N'$ dominates $\hyp(X_{m+1},N,\lf j/2\rf)$. To see why, let $N_1' = \#\{j' \le I(\lf j/2\rf)\colon \alpha(j') \in S\}$ which is $\F(I(\lf j/2\rf))$-measurable, and $N_2'=\#\{j' \in I(\lf j/2\rf)+1,\dots,I(j)\colon \alpha(j') \in S\}$. Then $N' = N_1' + N_2'$ and conditioned on $\F(I(\lf j/2\rf)$, $N_2'$ is $\hyp(X_{m+1}-I(\lf j/2\rf),|S|-N_1',I(j)-I(\lf j/2\rf))$. Now, $\hyp(M,K,n)$ is
\enumrom
\item stochastically increasing in $K$ (replace non-success objects with success objects),
\item stochastically increasing in $n$ (increase number of draws),
\item stochastically decreasing in $M$ (remove non-success objects), and
\item $m+\hyp(M-m,K-m,n-m)$ is stochastically increasing in $m$ \\(replace possible successes with certain successes).
\enumend
Moreover, $I(j)-I(\lf j/2\rf) \ge \lf j/2 \rf$, $N_1' \le I(\lf j/2\rf)$ and recall $|S|\ge N$. Using $\succeq$ to denote stochastic domination, conditioned on $\F(I(\lf j \rf/2)$
\begin{align*}
N'&=N_1'+'\hyp(X_{m+1}-I(\lf j/2\rf),|S|-N_1',I(j)-I(\lf j/2\rf)) \\
& \succeq \hyp(X_{m+1}-I(\lf j/2\rf)+N_1',|S|,I(j)-I(\lf j/2\rf)+N_1') \\
& \succeq \hyp(X_{m+1}-I(\lf j/2\rf)+N_1',|S|,I(j)-I(\lf j/2\rf)+N_1') \\
& \succeq \hyp(X_{m+1},N,\lf j/2\rf +N_1')\\
& \succeq \hyp(X_{m+1},N,\lf j/2\rf). 
\end{align*}
For convenience, let
$$Z=\hyp(\ell_m,L_m,\lf j/2\rf) \ \ \text{and} \ \ Z' = \hyp(X_{m+1},\lceil \ev[Z]/2 \rceil,\lf j/2\rf),$$
so that $N$ dominates $Z$ and, conditioned on $N\ge \lceil\ev[Z]/2 \rceil$, $N'$ dominates $Z'$. \\Note that $\ev[Z]=\lf j/2\rf L_m/\ell_m$ and
$$\ev[Z'] \ge \frac{\lf j/2\rf^2 L_m}{2\ell_m X_{m+1}}.$$
If $s_m^o>0$ then at least one litter has size $2$, so $\ell_m \le X_{m+1}-1$ and, recalling \eqref{eq:sntrunc},
\begin{align}\label{eq:evZ'}
\ev[Z'] \ge \frac{\lf j/2\rf^2 L_m}{2X_{m+1}(X_{m+1}-1)} \ge \frac{(j-1)^2 s_m^o}{16}.
\end{align}

We consider separately the cases $t_\infty^s=\infty$ and $t_\infty^\ell=\infty$.\\

\nid\tbf{Case 1: $t_\infty^s=\infty$.} Here we can mostly imitate the proof of the asynchronous case. Define $m(j)$ as before and let $W_j^s = t_{m(j+1)}^s - t_{m(j)}^s$. Then $t_M^s=\sum_{j=2}^{\infty} W_j$ so if we can show that $\sum_{j=2}^{\infty}\ev[\, W_j\,]<\infty$ then $M<\infty$ a.s.~which implies fixation. Let $q_m(j) = \pr(Y(m+1)>j \ \text{and} \ s_m^s>0\mid Y(m)=j)$. Then as before,
$$\pr(m(j+1)>a \mid m(j)=m) \le \exp(-\sum_{m'=m}^{a-1}q_a(j)).$$
Moreover, by definition, $s_m^s \le 2/(Y(m)(Y(m)-1))$ so $s_m^s \le 2/(j(j-1))$ for $m(j) \le m < M$. Suppose $Y(m)=j$ and let $N,N',Z,Z'$ be as above. If $s_m^s>0$ then since $Y(m+1)\ge Y(m)+N'$,
\begin{align}\label{eq:qlbZ}
q_m(j) \ge \pr(N'>0) \ge \pr(Z\ge \ev[Z]/2)\pr(Z'>0).
\end{align}
To get a lower bound, we'll use Lemma \ref{lem:pebnd}. If $\ev[Z] \le 18$, then $\pr(Z>\ev[Z]/2)\ge \lf j/2\rf L_m/(36\ell_m)$. If $s_m^s>0$ then $L_m>0$ so $\ev[Z]>0$ and, using $\hyp(X_{m+1},1,\lf j/2\rf)$ as a stochastic lower bound on $Z'$, $\pr(Z' >0) \ge \lf j/2\rf / X_{m+1}$. Plugging into \eqref{eq:qlbZ},
$$q_m(j) \ge \frac{\lf j/2\rf^2L_m}{36\,\ell_m X_{m+1}} .$$
Since $s_m^s>0$, at least one litter has size $2$, so $\ell_m \le X_{m+1}-1$. Noting that $\lf j/2\rf \ge (j-1)/2$ and recalling that $s_m^o = 2L_m/(X_{m+1}(X_{m+1}-1))$, it follows that
$$q_m(j) \ge (j-1)^2 s_m^s/288.$$

If instead $\ev[Z]>18$, then with $Z$ as above, $\pr(Z > \ev[Z]/2) \ge 1/2$. Since $Y(m)=j$ by assumption, $m(j) \le m < M$ so $j(j-1)s_m^s \le 2$ and $(j-1)^2s_m^s/16 \le 1/8$. Using \eqref{eq:evZ'},
$$\pr(Z' >0 ) \ge (\ev[Z'] \wedge 18)/36 \ge (j-1)^2 s_m^s/288$$ 
and combining with \eqref{eq:qlbZ}, $q_m(j) \ge \pr(Z>\ev[Z]/2)\pr(Z'>0) \ge (j-1)^2 s_m^s/576$.\\

Since in all cases, $q_m(j) \ge (j-1)^2 s_m^2/576$, and since in addition $s_m^s \le 2/(j(j-1))\le 2/(j-1)^2$ for $m(j) \le m < M$, proceeding as in the asynchronous case we find that $\ev[\, W_j \,] \le C/(j-1)^2$ for some $C>0$, and conclude that $\sum_{j=2}^{\infty} \ev[\, W_j\,]<\infty$. \\

\nid\tbf{Case 2: $t_\infty^\ell=\infty$. }Here we need Lemmas \ref{lem:divseq} and \ref{lem:binsum} as well as \ref{lem:pebnd}. Let
$$\eta_m=\begin{cases} \1(Y(m+1) \ge Y(m) + (Y(m)-1)^2s_m^\ell/16) & \text{if} \ \ m<M, \\
1 & \text{if} \ \ m\ge M,\end{cases}$$
$J_m = \sum_{m'=n}^{m-1}\eta_{m'}$ and let $m(0)=n$ and $m(i)= \inf\{m\colon J_m=i\}$. If $m(i)\ge M$ then $m(i+a)=m(i)+a$, and if $m<M$, then since $m\mapsto Y(m)$ is non-decreasing, $Y(m(i+1)+1) \ge Y(m(i)+1) + (Y(m(i)+1)-1)^2 s_{m(i+1)}^\ell/16$, which we will use in a moment.\\

Let $q_m(j) = \pr(\eta_m=1\mid Y(m)=j)$. With $N,N',Z,Z'$ as above and using \eqref{eq:evZ'},
\begin{align}\label{eq:qlb3}
q_m(j) \ge \pr(N' \ge (j-1)^2 s_m^\ell/16) \ge \pr(Z > \ev[Z]/2)\pr(Z' > \ev[Z']/2).
\end{align}
By assumption, either $s_m^\ell=0$ or $s_m^\ell \ge 2/(Y(m)(Y(m)-1))$ for each $m$. \\If $Y(m)=j$ and $s_m^\ell>0$, then 
\[\frac{2L_m}{X_{m+1}(X_{m+1}-1) } = s_m^o = s_m^\ell \ge \frac{2}{j(j-1)} \ge \frac{2}{(j-1)^2}.\]
Using $\ev[Z]=\lf j/2\rf L_m/\ell_m$ and the above inequality,
\[\ev[Z] \ge \frac{(j-1) L_m}{2X_{m+1}} \ge \frac{X_{m+1}-1}{2(j-1)}.\]
If $m<M$ then $j=Y_m\le \ell_m$, and $\ell_m\le X_{m+1}$, so combining with the above we find $\ev[Z] \ge \frac{1}{2}$, so $\pr(Z>\ev[Z]/2) \ge 1/72$. Using \eqref{eq:evZ'} and the above, if $s_m^\ell>0$ then
$$\ev[Z'] \ge \frac{(j-1)^2s_m^\ell }{16} \ge \frac{1}{8},$$
and so $\pr(Z' > \ev[Z']/2) \ge 1/288$. Combining and plugging into \eqref{eq:qlb3}, if $s_m^\ell>0$ then $q_m(j) \ge p$ where $p=1/(72\cdot 288)$. \\

Let $(\F_m)$ be the natural filtration of $(\E_{m-1})$. Since $(\E_m)$ are independent, $\pr(\eta_m=1 \mid Y(m)=j, \ \F_m)=q_m(j)$, so for $m<M$
$$\pr(\eta_m=1\mid \F_m) = \ev[\, q_m(Y(m)) \mid \F_m] \ge p.$$
Since $\eta_m=1$ for $m\ge M$, it follows that $(\eta_m)_{m\ge n}$ dominates an i.i.d.~Bernoulli$(p)$ sequence $(u_m)_{m\ge n}$. Since $\sum_m s_m^\ell=\infty$ by assumption and $\{m(i)\colon i\ge 0\}=\{m\colon \eta_m=1\}$, using Lemma \ref{lem:binsum}, $\sum_i s_{m(i)}^\ell=\infty$. Let $y_i = Y(m(i)+1)-1$, $c=1/16$, $s_i=s_{m(i+1)}^\ell$, $t_i=\sum_{a\le i}s_a$ and $I = \inf\{i\colon y_i=\infty\}$. Earlier we showed that $y_{i+1} \ge y_i +  c\,y_i^2 s_i$, and for $m<M$, using $2L_m\le X_{m+1}$ and $Y(m) \le \ell_m \le X_{m+1}$,
$$s_m^\ell \le s_m^o=\frac{2L_m}{X_{m+1}(X_{m+1}-1)} \le \frac{1}{Y(m)-1},$$
so $s_i = s_{m(i+1)}^\ell \le 1/(Y(m(i+1))-1)$, and since $m(i+1)\ge m(i)+1$ and $m\mapsto Y(m)$ is non-decreasing, $1/(Y(m(i+1))-1) \le 1/(Y(m(i)+1)-1) = 1/y_i$, so $s_i \le 1/y_i$. Using Lemma \ref{lem:divseq}, $t_I<\infty$. Since $\sum_i s_i=\infty$, $I<\infty$, and $M \le m(I)<\infty$ a.s.~as desired.
\end{proof}

\section{Identifiability}\label{sec:ip}

\begin{proof}[Proof of Theorem \ref{thm:iptauinf}]
We prove \eqref{eq:baseip} first, since it is the simplest. For any $w\in V_n$, $p_m(w) = x_m'(w)$ since frequencies sum to $1$. By Lemma \ref{lem:xmart}, $x_m'(w)\to x(w)$ as $m\to\infty$ and by Lemma \ref{lem:rcm}, $p_m(w)\to \mu((n,1),w)$. Using \eqref{eq:ip},
\[\rho((n,1))=\ev[\max_{w\in V_n}x(w)].\]

Next we prove existence of $x(w,u)$. Since $\E,\E'$ are isomorphic, it can be proved using either one as a reference; we will use $\E'$. Letting $(w_i)$ range over the finitely many enumerations of $V_n$,
\[\sum_{(w_i)} \pr( r(w_i)=i \ \forall i\le X_n \mid \E')=1\]
so for some $(w_i)$, that we fix for now, $\pr( r(w_i)=i \ \forall i\le X_n \mid \E')>0$, and with $p_\ell= \pr( r(w_i)=i \ \forall i\le \ell \mid \E')$, $p_\ell \ge p_{X_n}>0$ for all $\ell\in\{1,\dots,X_n\}$. Let
\[q_m(i)=\frac{x_m'(w_i)}{\sum_{j\ge i}x_m'(w_j)} \in [0,1],\]
so that $p_\ell=\lim_{m\to\infty} \prod_{i=1}^\ell q_m(i)$ and $q(i):=\lim_{m\to\infty} q_m(i)=p_i/p_{i-1} \in (0,1]$ for each $i$. Let $a_m(i)=x_m'(w_i)$ and suppose for a given $i$ and all $i<j<k\le X_n$ that
\[q(k,j):=\lim_{m\to\infty} a_m(k)/a_m(j) \in [0,\infty).\]
Then
\[\lim_{m\to\infty}\frac{a_m(i)}{a_m(i)+\dots + a_m(X_n)} =q(i)\in (0,1]\]
so
\begin{align*}
1/q(i)&=\lim_{m\to\infty} \frac{a_m(i)+\dots +a_m(X_n)}{a_m(i)} \\
&= 1 + \lim_{m\to\infty} \frac{a_m(i+1)}{a_m(i)}\left( 1 + \frac{a_m(i+2)}{a_m(i+1)} + \dots + \frac{a_m(X_n)}{a_m(i+1)}\right)\\
&= 1 + \left(1 + q(i+2,i+1)+\dots q(X_n,i+1)\right)\lim_{m\to\infty} \frac{a_m(i+1)}{a_m(i)}
\end{align*}
and since $1/q(i)\in [1,\infty)$ and $q(k,i+1)\in[0,\infty)$ for $i+1<k\le X_n$, $\lim_{m\to\infty} a_m(i+1)/a_m(i)=:q(i+1,i) \in [0,\infty)$, from which $\lim_{m\to\infty} a_m(k)/a_m(i) =: q(k,i) = q(k,i+1)q(i+1,i)\in [0,\infty)$ for $i<k\le X_n$, completing the induction step. The base case $i=X_n-1$ holds, since it reduces to
\[1/q(X_n)=1 + \lim_{m\to\infty} \frac{a_m(X_n)}{a_m(X_n-1)}=: 1 + q(X_n,X_n-1)\]
and $1/q(X_n)\in [1,\infty)$. The existence of $x'(w,u)$ (and thus $x(w,u)$) is proved, since each $(w,u)=(w_i,w_j)$ for some $i,j$, and \[x'(w_i,w_j)=\begin{cases}
q(i,j)\in [0,\infty) & \text{if} \ i>j, \\
1/(q(j,i))\in (0,\infty] & \text{if} \ i<j.\end{cases}\]

Finally we prove \eqref{eq:tauinfip}. With $v_i=(n,i)$ let $c((w_i)_{i=1}^\ell)=1/\mu((v_i)_{i=1}^\ell,(w_i)_{i=1}^\ell)$. Lemma \ref{lem:rcm} implies $c((w_i)_{i=1}^\ell)=\lim_{m\to\infty} 1/p_m((w_i)_{i=1}^\ell$. On the other hand, \eqref{eq:rcm} gives
\[1/p_m((w_i)_{i=1}^\ell) = \prod_{i=1}^\ell \left(1 + \sum_{j=i+1}^{X_n} \frac{x_m'(w_j)}{x_m'(w_i)}\right)\]
so combining these with \eqref{eq:limrelfreq} we find that
\[c((w_i)_{i=1}^\ell)=\prod_{i=1}^\ell \left(1 + \sum_{j=i+1}^{X_n} x'(w_j,w_i)\right).\]
By \eqref{eq:ip2},
\[\rho((n,i)_{i=1}^\ell) = \ev[1/\min_{(w_i)}c((w_i)_{i=1}^\ell)],\]
the minimum taken over enumerations $(w_i)$ of $V_n$, so it suffices to show that $c((w_i)_{i=1}^\ell$ is minimized by any $(w_i)$ for which $x'(w_j,w_i)\le 1$ for $j\ge i$; of course, whether we use $x(w,u)$ or $x'(w,u)$ is immaterial since they coincide up to permutation of the entries. Since $x(u,w)=1/x(w,u)$, there is an ordering for which $x(w_j,w_i)<\infty$ for $j\ge i$, and $c((w_i)_{i=1}^\ell)<\infty$ iff this holds, so any minimizer must have this property. Switch notation back to $q(j,i)=x(w_j,w_i)$ for tidiness. It suffices to show that if $q(j,i)>1$ for some $j>i$ then the value of $c$ can be decreased by re-ordering some terms. Indeed, if this occurs then since $q(j,i)=q(j,j-1)\dots q(i+1,i)$, $q(k+1,k)>1$ for some $k$. Consider the effect of exchanging $w_k$ and $w_{k+1}$ on $c$. The only affected terms in the product are $i\in\{k,k+1\}$. Before the exchange, the product of those two terms is
\[(1 + q(k+1,k) + q(k+2,k) + \dots + q(X_n,k))(1 + q(k+2,k+1) + \dots + q(X_n,k+1))\]
and after the exchange it's
\[(1 + q(k,k+1) + q(k+2,k+1) + \dots + q(X_n,k+1))(1 + q(k+2,k) + \dots + q(X_n,k)).\]
Taking the difference, most terms cancel and the result is
\begin{align*}
&q(k+1,k)(1+q(k+2,k+1)+\dots + q(X_n,k+1)) \\
-& q(k,k+1)(1 + q(k+2,k) + \dots + q(X_n,k)) \\
= &(1-q(k,k+1))(1+q(k+2,k)+\dots + q(X_n,k)),\\
\end{align*}
using $q(\cdot,k+1)q(k+1,k)=q(\cdot,k)$. By assumption, $q(k+1,k)>1$ so $q(k,k+1)=1/q(k+1,k)<1$ and the difference is positive, meaning $c$ decreases after the exchange.
\end{proof}

\begin{proof}[Proof of Corollary \ref{cor:baseip}]

Let $u_n=(n,1)$. As noted above the statement of Corollary \ref{cor:baseip} it suffices to show that $\rho((u_n)_{n\ge 0})=\lim_{n\to\infty} \rho(u_n)$. Since, for $i\le n$, $u_i$ is the ancestor of $u_n$ in generation $i$,
\[\max_{(w_i)_{i \le n} \in V_0\times \dots\times V_n \colon w_n=w} \mu_{(u_i)_{i \le n},(w_i)_{i\le n}} = \mu_{u_n,w}\]
with the maximum achieved by taking $w_i=a(w_{i+1})$ for $i<n$. Taking $\ev[\max_{w\in V_n}(\dots)]$ of both sides we find that $\rho((u_i)_{i\le n}) = \rho(u_n)$.\\

Next, if $m<n$ then for any $(v_i)_{i\le n}$, $(w_i)_{i\le n}$,
\[\mu_{(v_i)_{i\le m},(w_i)_{i\le m}} \le \mu_{(v_i)_{i\le n},(w_i)_{i \le n}}\]
and in particular, for fixed $(v_n)_{n\ge 0}$, $\max_{(w_i)_{i\le n}}\mu_{(v_i)_{i\le n},(w_i)_{\le n}}$ is non-increasing in $n$, so converges a.s.~as $n\to\infty$. The bounded convergence theorem then implies that $\rho((v_n)_{n\ge 0}) = \lim_{n\to\infty} \rho((v_i)_{i\le n})$. Taking $v_n=u_n$ and combining with the above we obtain the desired result.
\end{proof}

\end{document}